\documentclass{article}
\usepackage{lineno,hyperref}
\usepackage{amsfonts}
\usepackage{amsmath}
\usepackage{latexsym,amssymb}
\usepackage{amsmath}
\usepackage{indentfirst}
\usepackage{tikz}
\usepackage{euscript}
\usepackage[all]{xy}
\usepackage{empheq}
\usepackage{fancybox}
\usepackage{framed}
\usepackage[square,sort,comma,numbers]{natbib}
\usepackage{bibentry}
\usepackage{color}

\newtheorem{theorem}{Theorem}[section]
\newtheorem{lemma}[theorem]{Lemma}
\newtheorem{example}[theorem]{Example}
\newtheorem{remark}[theorem]{Remark}

\newtheorem{definition}[theorem]{Definition}

\usepackage{extarrows}
\allowdisplaybreaks

\def \Ra {\xRightarrow}

\def \k {\kappa}
\def \d {\partial}
\def \s {\scriptstyle}
\def \A {{\cal A}}
\def \ra{\xrightarrow}

\input xypic
\input xy
\xyoption{v2}
\xyoption{all}
\xyoption{2cell}

\makeatother
\newcommand{\pmor}[2] {#1 \xrightarrow{ \ \ \ \ #2 \ \ \ } \left(#1+\partial_{1}(#2)\right)}






\newcommand{\trinearlyn}[4] {& & #4+\partial_{1}(#3)+\partial_{1}(#2)\\ & & \quad \quad \quad \\
	#4  \ar[rruu]^{#3+#2}\ar[rr]_{#3} & & {#4+\partial_{1}(#3)} \ar[uu]_{\s{#2 +\partial_{2}(#1)}}^{\ovalbox{$#1$}}}

\newcommand{\trinearlyy}[7] {& & \s{#4} \\ & & \quad \quad \quad \\
	\s{#2}  \ar[rruu]^{\s{#7}}\ar[rr]_{\s{#5}} & & {\s{#3}}  \ar[uu]_{\s{#6}}^{\ovalbox{$\s{#1}$}}}
\makeatother

\newcommand{\trianglyyy}[7] {& & \s{#3} \\ & & \\
	\s{#1}  \ar[rruu]^{\s{#6}} \ar[rrr]_{\s{#4}}^{\ovalbox{$\s{#7}$}}\quad & & &  {\s{#2}}  \ar[uul]_{\,\,\,\s{#5}}}





\newcommand{\free}[5]{ \s{#1} \ \ar@{^{(}->}[r]^{\s{incl}} \ar[dr]_{\s{#4}} & \s{#2} \ar@{.>}[d]^{\s{#5}}_{\s{\exists !}} \\ & \s{#3}}





\newcommand{\tetrnn}[7] {
	& & #1+\partial_{1}(#2)+\partial_{1}(#3)+\partial_{1}(#5) & & \\
	& & & & \\
	& & \quad \,\, #1+\partial_{1}(#2)+\partial_{1}(#3) \ar[uu]|{#5+\partial_{2}(#6)+\partial_{2}(#7)} & & \\
	#1 \ar[rrrr]_{#2}^{\ovalbox{$#4$}} \ar[rru]_{#2+#3} \ar[rruuu]^{#2+#3+#5}_{\ovalbox{$#6+#7$}} & & & & #1+\partial_{1}(#2) \ar[llu]^{#3+\partial_{2}(#4)} \ar[lluuu]_{\ \ \  #3+#5+\partial_{2}(#4)+\partial_{2}(#6)}^{\ovalbox{$#7$}}
}


\newcommand{\xxtrinearly}[4] {& & \s{#4 +\partial_{1}'(#3) +\partial_{1}'(#2)}\\ & & \quad \quad \quad \\
	\s{#4}  \ar[rruu]^{\s{#3+#2}}\ar[rr]_{\s{#3}} & & {\s{#4}+\partial_{1}'(#3)} \ar[uu]_{\s{#2 +\partial_{2}'(#1)}}^{\ovalbox{$\s{#1}$}}}

\newcommand{\xtrinearlyy}[7] {& & \s{#4} \\ & & \quad \quad \quad \\
	\s{#2}  \ar[rruu]^{\s{#7}}\ar[rr]_{\s{#5}} & & {\s{#3}}  \ar[uu]_{\s{#6}}^{\ovalbox{$\s{#1}$}}}
\makeatletter
\newenvironment{proof}[1][\emph{Proof}]{\textbf{#1:} }{\ \rule{0.5em}{0.5em}}

\def \G {\mathcal{G}}

\begin{document}

\title{Two-Fold Homotopy of 2-Crossed  Module Maps of Commutative Algebras}

\author{ 
\.{I}.\.{I}lker Ak\c{c}a \\ {\it \small i.ilkerakca@gmail.com}\\
{   \small Department of Mathematics and Computer Science},\\ { \small Eski\c{s}ehir Osmangazi University, Turkey.}
\\ \quad \\
Kad\.{i}r Em\.{i}r\thanks{Kadir Emir was collaborating member of CMA/FCT/UNL, 2014 /15, and acknowledges support from the grant:
	FCT/Portugal Strategic Project UID/MAT/00297/2013.} \\ {\it  \small kadiremir86@gmail.com} \\ {   \small Department of Mathematics and Computer Science},\\ { \small Eski\c{s}ehir Osmangazi University, Turkey.} 
 \\ \quad \\ 
Jo\~ao Faria Martins \\{\it  \small j.fariamartins@leeds.ac.uk} \\ { \small School of Mathematics, University of Leeds}\\{ \small Leeds, LS2 9JT, UK.} }

\maketitle

\begin{abstract}
	We address the homotopy theory of 2-crossed modules of commutative algebras. In particular, we define the concept of a 2-fold homotopy between a pair of 1-fold homotopies connecting 2-crossed module maps $\A \to \A'$. We also prove that if the domain 2-crossed module $\A$ is free up to order one (i.e. if the bottom algebra is a polynomial algebra) then we have a 2-groupoid of 2-crossed module maps $\A \to \A'$ and their homotopies and 2-fold homotopies. 
\end{abstract}

\bigskip

\noindent{\bf Keywords:} {Simplicial commutative algebra, crossed module of commutative algebras, 2-crossed module of commutative algebras, quadratic derivation.}

\medskip

\noindent{\bf 2010 AMS Classification:} {55U10 (principal), %Simplicial sets and complexes
 18D05,    %Double categories, 2-categories, bicategories and generalizations }
18D20,  %Enriched categories (over closed or monoidal categories)
55Q15    %Whitehead products and generalizations
(secondary). }

\newpage

\section{Introduction}
Group crossed modules  $\G=(\d\colon E \to G, \blacktriangleright)$  are given by a group homomorphism $\d\colon E \to G$, together with an action $\blacktriangleright$ of $G$ on $E$ by automorphisms, such that the Peiffer relations, displayed below, hold for all  $e,f \in E$ and $g \in G$: 
\begin{align*}
\textrm{ \bf{Peiffer 1}: }\, \d(g \blacktriangleright e)& =g \,\d(e) \, g^{-1}; &&  \textrm{\bf{ Peiffer 2}: }\, \d(e)  \blacktriangleright f =e\, f \, e^{-1}.
\end{align*}
2-crossed modules of groups are given by a complex \smash{$\A=(L \ra{\delta} E \ra{\d} G, \blacktriangleright,\{,\})$} of groups, together with actions $\blacktriangleright$ of $G$ on $L$ and $E$, by automorphisms, making it a complex of $G$-modules, where $G$ acts on itself by conjugation. The first Peiffer relation for the map $\d\colon E \to G$ and the action $\blacktriangleright$ of $G$ on $E$  therefore automatically holds, thus  $(\d\colon E \to G, \blacktriangleright)$ is what is called a pre-crossed module. The second Peiffer relation does not hold in general. However we have a map $\{,\}\colon E \times E \to L$, called the Peiffer lifting, measuring how far Peiffer 2 is from being satisfied, namely: $\delta(\{e,f\})=\big(efe^{-1}\big) \, \big( \d(e) \blacktriangleright f^{-1}\big)$, for all $e,f \in E$. 

\medskip

The category of 2-crossed modules is equivalent to a reflexive subcategory of the category of simplicial groups \cite{C1,MP1,MP2}. Thus \cite{CG,GM1}  an adjunction: $$
\xymatrix{ \{\textrm{2-crossed modules of groups} \}
	\ar@/^1pc/[rr]|{{B}}="1"
	\ar@/_1pc/@{<-}[rr]|{\Pi_3}="2"
	&& 
	\ar@{}"1" ;"2"^{\top}
	\{\textrm{Simplicial sets}\}},
$$
exists. Such is constructed  in the obvious way from the Dwyer-Kan loop-group $G \dashv \overline{W}$ adjunction between the category of simplicial sets and the category of simplicial groups; \cite{Q1,DW,GJ1}. ${B}$ is the simplicial classifying space functor \cite{B1} and $\Pi_3$ was explicitly constructed in \cite{M1} from triad homotopy groups of the geometric realisation. We thus \cite{CG1,GJ1} have a model category structure in the category of 2-crossed modules of groups in which all objects are fibrant and  which renders a 2-crossed module  $\A=(L \to E \to G, \blacktriangleright,\{,\})$ cofibrant if, and only if, it is a retract of a totally free 2-crossed module, the latter meaning that $G$ is a free group and that $(\d\colon E \to G, \blacktriangleright)$ is a free pre-crossed module \cite{MP1,MP2,M1,GM1}.

Following previous constructions in the context of quadratic modules (a particular case of 2-crossed modules) \cite{B1} and Gray categories  \cite{Crans}, 
a homotopy relation between group 2-crossed module maps was addressed in \cite{G1,GM1}, via path-objects, and proven in \cite{M1} to faithfully model homotopy classes of maps between 3-types.   Given 2-crossed modules $\A$ and $\cal B$, of groups, if $\A$ is totally free (therefore cofibrant), if follows that homotopy between maps $\A \to \cal B$ is an equivalence relation.  A surprising result of \cite{G1,GM1} was that, in order for homotopy between maps $\A \to \cal B$ to be an equivalence relation, we solely need to impose that $\A$ is free up to order one, meaning that $G$ is a free group. Moreover in the latter case we can define, if we are given a free basis of $G$, a 2-groupoid of maps $\A \to \cal B$, homotopies between maps, and 2-fold homotopies between homotopies.

This paper, which is a follow-up of \cite{IJK}, contains a proof that the latter property (1-freeness suffices to compose homotopies and 2-fold homotopies) also holds for 2-crossed modules of commutative algebras.

All algebras in this paper will be commutative and over a fixed ring $\kappa$, not necessarily with 1. Crossed modules and 2-crossed modules of algebras \cite{AP1,AP2,AP3,DGV1,P1} are defined in the same way as in the group case, essentially switching actions by automorphisms to actions by multipliers, where a multiplier in an algebra $R$ is a linear map $f\colon R \to R$ such that $f(ab)=f(a)b$, for each $a,b \in R$.  

A  2-crossed module of algebras  {$\A=(L\overset{\partial _{2}}{\longrightarrow }E\overset{\partial _{1}}{ \longrightarrow } R, \blacktriangleright,\{,\})$} has an underlying complex of algebras and some other  data: we  have  actions $\blacktriangleright $ by multipliers of $R$  on $E$ and $L$ and a  map $\{,\}\colon E \times E \to L$ measuring how far $(\d_1\colon E  \to R\blacktriangleright)$ is from being a crossed  module of algebras. Therefore $\d_2(\{e,f\})=ef-\d_1(e)\blacktriangleright  f$ for all $e,f \in E$. 

As in the group case  \cite{C1,MP1,MP2}, simplicial algebras and 2-crossed modules of algebras are closely related. A simplicial algebra \cite{DGV1,AP1,AP2,Andre,Moerdijk} $A=(A_n,d_n^i,s_n^i)$, i.e. a simplicial object in the category of algebras,  is given by a collection $A_n$ of algebras and algebra maps $d_n^i \colon A_n \to A_{n-1}$, $i=0, \dots,n$, and $s_n^i\colon A_n \to A_{n+1}$,  $i=0, \dots,n$, respectively called boundaries and degeneracies, and satisfying the well known simplicial identities. The Moore complex:
$$N(A)= \big( \dots \ra{d_{(n+1)}}  N(A)_n \ra{d_n} \dots  \ra{d_{3}} N(A)_2 \ra{d_2} N(A)_1 \ra{d_1} A_0  \big)$$ 
of the simplicial algebra $A$ has the algebra \smash{$N(A)_n=\bigcap_{i=0}^{n-1} \ker(\d_n^i) \subset A_n$} at level $n$ and the boundary $d_n\colon N(A)_n \to N(A)_{n-1}$ is the restriction of  $d_n^n\colon A_n \to A_{(n-1)}$. We say that the Moore complex of a simplicial algebra $A$ has length $n$ if $N(A)_i$ is trivial for $i > n$.

\medskip

If $A$ has Moore complex of length two then  \smash{$(N(A)_2 \ra{d_2} N(A)_1 \ra{d_1} A_0)$} is naturally a 2-crossed module. This gives \cite{A1,GV1,P1} an equivalence of categories from the category of simplicial algebras with Moore complex of length two and the category  of 2-crossed modules of algebras. 

\medskip

An algebra 2-crossed module {$\A=(L{\to }E { \to } R, \blacktriangleright,\{,\})$}  is called free up to order one if  $R$ is free, over a set. We will normally consider a particular free basis $B \subset R$. Thus $R$ should be seen as an algebra of commutative polynomials, with a formal variable assigned to each element of $B$. An analogous definition which would render all results of the paper correct, {\it mutatis mutandis}, would be to make $R$ free over the $\kappa$-module $V$, and chose an isomorphism $R \to S(V)$, where $S(V)$, the symmetric algebra, is the free commutative algebra over $V$.

\medskip

If  $\A=(L{\to}E{ \to } R, \blacktriangleright,\{,\})$ and  $\A'=(L'{\to}E'{ \to } R', \blacktriangleright,\{,\})$ are 2-crossed modules of algebras, then 2-crossed module maps $f\colon A \to A'$ are defined in the obvious way. The homotopy relation between maps of 2-crossed modules of algebras was addressed in \cite{IJK}. Consider two maps $f,g\colon \A \to \A'$. A homotopy $f \to g$ is given by a pair of set maps $\big (s\colon R \to E', t \colon E \to L'\big)$, satisfying further relations (strongly depending on $f$ and therefore on $g$),  defining what we called a quadratic $f$-derivation $(s,t)\colon \A \to \A'$. As in the group case, the homotopy relation between 2-crossed module maps $\A \to \A'$ is only an equivalence relation, in general, in the case when $\A$ is free up to order one. 

\medskip

Presumably, retracts of free up to order one 2-crossed modules will be the cofibrant objects of a model category structure in the category of 2-crossed modules of algebras, which is yet to be discovered. This model category will be different from the model structure naturally derived from the category of simplicial sets: \cite{Moerdijk,GS,Q1}. The latter renders a 2-crossed module of algebras $\A=(L{\to}E{ \to } R, \blacktriangleright,\{,\})$  cofibrant if and only if it is a retract of a 2-crossed module which is free up to order two (\cite{IJK,GM1}), meaning that not only the algebra $R$ is free but also the pre-crossed module $(\d\colon E \to R, \blacktriangleright)$ is free; \cite{AP1}. 

\medskip

It is proven in \cite{IJK} that if a chosen basis is given for the free algebra $R$, then we can define a groupoid ${\rm HOM}(\A,\A')$, whose objects are the 2-crossed module maps $\A \to \A'$,  the morphisms being the homotopies  (also called 1-fold homotopies) between them. (We  will review here the most important issues of this construction, which are  essential for understanding the main aspects of this article.) As hinted above, in this paper we go one step further, and consider 2-fold homotopy. This 2-fold homotopy relation is a relation connecting homotopies of 2-crossed module maps $\A \to \A'$. A 2-fold homotopy is given by a linear map $q\colon R \to L'$, satisfying further natural properties, defining what we called a quadratic-2-derivation. The main result of this paper is that if $\A$ is free up to order one, with a chosen basis for $R$, then we have a  2-groupoid ${\rm HOM}(\A,\A')_{2}$  of 2-crossed module maps $\A \to \A'$,  their homotopies, and 2-fold homotopies connecting 2-crossed 
module homotopies. 

\medskip

The proof of this result requires quite a few complicated calculations. As in \cite{IJK}, mostly all difficult calculation are done while defining the algebras of 1, 2 and 3-simplices in a 2-crossed module $\A$, which are very important constructions on their own. All of the compositions  in the 2-groupoid ${\rm HOM}(\A,\A')_{2}$ will in one way or another be derived from the boundaries of these  geometrically constructed algebras.  

\medskip
Some issues we expect to investigate in future publications are:

\begin{itemize}
	\item The homotopies we considered  for 2-crossed module maps $\A \to \A'$  are pointed (they are algebra analogues of pointed homotopies between maps between crossed modules of groups \cite{BHS}). A full homotopy theory should take into account that base-points may move.  Given 2-crossed modules $\A=(L{\to}E{ \to } R, \blacktriangleright,\{,\})$ and  $\A'=(L'{\to}E'{ \to } R', \blacktriangleright,\{,\})$, with $\A$ free up to order one, and crossed module maps $f,f'\colon \A \to \A'$, a full quadratic derivation $(s,t,b)$, connecting $f$ to $f'$, is given by a quadratic derivation $(s,t)$, connecting $f$ to $g$ ($g\colon \A \to \A'$ is an intermediate map), as well as an idempotent element of $b \in R'$, such that $f'=b\ \blacktriangleright g$. This extra layer of structure leads to the definition of a Gray-3-groupoid \cite{Crans,KP} of maps $\A \to \A'$, and their 1-, 2- and 3-fold homotopies. We expect to give full formulae in a future publication. (This requires a large amount of new calculations.)
	\item Let $\A=(L{\to}E{ \to } R, \blacktriangleright,\{,\})$ be a 2-crossed module of algebras, not necessarily  free up to order one. We have a cofibrant replacement monad $Q^1$ \cite{GM1,G1}, sending $\A$ to a free up to order one 2-crossed module, together with a map $p\colon Q^1(A) \to \A$, inducing isomorphism of homology algebras. As in \cite{GM1}, the associated co-Kleisli category (a map $\A \not \to \A'$ in the co-Kleisli category is a 2-crossed module map $Q^1(\A) \to \A'$) enlarges our collection of maps from $\A$ to $\A'$, as well as homotopies and 2-fold homotopies between them, in such a way that any pair of 2-crossed modules $\A$ and $\A'$ yields a 2-groupoid   
	${\rm HOM}_{\rm LAX}(\A,\A')_{2}={\rm HOM}(Q^1(\A),\A')_{2}$ of lax (co-Kleisli) maps $\A \not\to \A'$ and (lax) homotopies and 2-fold homotopies between them. This comonadic approach furthermore implies that given any lax map $\A'\not \to \A''$ then we have a whiskering 2-functor $${\rm HOM}_{\rm LAX}(\A,\A')_{2} \to {\rm HOM}_{\rm LAX}(\A,\A'')_{2}.$$ We conjecture that this means that the co-Kleisli category of $Q^1$ can be enriched to be a Gray category.  
	\item Let $\kappa$ be a ring. Given a set $X$ consider the polynomial algebra $\kappa[X]$ over $X$. Given an ideal $I$ of $\kappa[X]$ we have a crossed module of algebras given by the inclusion map $I \to \kappa[X]$. This crossed module can be turned into a 2-crossed module in two different ways: as $\A_0(I,X)=(\{0\}\to I \to \kappa[X])$ or as $\A_1(I,X)=(\{I \otimes I\} \to I \to \kappa[X],\{,\})$. The latter is the free 2-crossed module on the map $\emptyset \to I$ (see \cite{AP1}). Thus $\{I \otimes I\}$ is generated by the symbols $\{i\otimes j\}$ satisfying the obvious relations implied by the properties of Peiffer liftings. Given any 2-crossed module $\A'$ of commutative algebras, we thus have 2-groupoids  ${\rm HOM}(\A_0(I,X),\A')_{2}$ and ${\rm HOM}(\A_1(I,X),\A')_{2}$. These are invariants of the ideal $I$ and in particular given 2-crossed module maps $f_k\colon \A_k(I,X) \to \A'$ this gives groups $\pi_1({\rm HOM}(\A_k(I,X),\A')_{2},f_k)$ and $\pi_2({\rm HOM}(\A_k(I,X),\A')_{2},f_k)$, where $k=0,1$, which likely have some algebraic geometric interpretation. 
\end{itemize}

% 	\medskip 
% 	\item Choose elements $(x_1,\dots,x_n)$ of $\kappa[X]$. Form \cite{AP1} the free crossed module $$\d\colon F\big(x_1,\dots,x_n|\kappa[X]\big) \to \k[X] $$ on the inclusion map $\{x_1,\dots,x_n\} \to \k[X]$. Upgrade it into a 2-crossed module $ \big(\ker(\partial) \to F\big(x_1,\dots,x_n|\kappa[X]\big) \to \k[X] \big) $, with trivial Peiffer pairing. (As observed in \cite{P1},  $\ker(\partial)$ is the first homology of the Koszul complex determined by $\{x_1,\dots,x_n\} \subset \kappa[X]$.) As above given a 2-crossed module $\A'$ we have a 2-groupoid   ${\rm HOM}(  \ker(\partial) \to F\big(x_1,\dots,x_n|\kappa(X)\big) \to \k[X] , \A')_{2}$. The homotopy  groups of the latter are an invariant of the pair $\{x_1,\dots,x_n\} \subset \kappa[X]$. In a future publication we expect to address applications to Kozsul and Andr\'{e}-Quillen homology. 

\section{Preliminaries}
We follow the conventions of \cite{A1,AP1,AP2,AP3}. In this paper we fix a commutative ring $\k$, not necessarily with 1. All algebras considered will be associative and commutative, but not necessarily with a multiplicative identity. If $M$ and $R$ are $\k$-algebras, a bilinear map
$(r,m) \in R\times M \longmapsto  r\blacktriangleright m \in M$
is called an algebra action of $R$ on $M$ if, for all $m,m^{\prime }\in M$ and $r,r^{\prime }\in R$, we have:
$$
{\bf A1}\colon \, r\blacktriangleright (mm^{\prime })=(r\blacktriangleright
m)m^{\prime }=m(r\blacktriangleright m^{\prime }), \quad
{\bf A2}\colon \, (rr^{\prime })\blacktriangleright m=r\blacktriangleright
(r^{\prime }\blacktriangleright m).
$$
% 
% \begin{description}
% 	\item[A1.] $r\blacktriangleright (mm^{\prime })=(r\blacktriangleright
% 	m)m^{\prime }=m(r\blacktriangleright m^{\prime })$,
% 	\item[A2.] $(rr^{\prime })\blacktriangleright m=r\blacktriangleright
% 	(r^{\prime }\blacktriangleright m)$.
Our convention for the commutative algebra semidirect product $R\ltimes
_{\blacktriangleright }M$ is:
\begin{equation*}
(r,e) \, (r^{\prime },e^{\prime })=(rr^{\prime },r\blacktriangleright e^{\prime }+r^{\prime
}\blacktriangleright e+ee^{\prime }), \textrm { for all } e,e^{\prime }\in M \textrm { and }r,r^{\prime}\in R.
\end{equation*}

\subsection{2-crossed modules of commutative algebras}

\begin{definition}[Pre-crossed modules and crossed modules]
	A pre-crossed module $(E,R,\partial )$, of (commutative) algebras, is given by an algebra map $\partial\colon E \to R$, with $E$ and $R$ commutative, together with an action $\blacktriangleright $ of $R$ on $E$, such that
	the following relation, called the \textquotedblleft first Peiffer relation\textquotedblright, holds:
	\begin{description}
		\item[XM1)] $\partial (r\blacktriangleright e)=r \, \partial (e)$, for each $e \in E$ and $r \in R$.
	\end{description}
	A crossed module of  algebras  $(E,R,\partial )$ is a pre-crossed
	module satisfying, furthermore, the \textquotedblleft second Peiffer relation\textquotedblright:
	\begin{description}
		\item[XM2)] $\partial (e)\blacktriangleright e^{\prime }=e \, e'$, for all $e,e^{\prime }\in E$.
	\end{description}
\end{definition}

\begin{example}
	Let $R$ be an algebra and $E\trianglelefteq R$ an ideal. Then $(E,R,i)$, where $i\colon E \to R$ is the inclusion map, is a
	crossed module. We use the multiplication in $R$ to define the action of $R$ on $E$: 
	$(r,e) \in R\times E \longmapsto  r\blacktriangleright e=re \in E.$
\end{example}
\begin{definition}\label{2xmodule}
	A {\bf 2-crossed module}   \smash{$\A=(L\overset{\partial _{2}}{\longrightarrow }E\overset{\partial _{1}}{ \longrightarrow } R, \blacktriangleright,\{,\})$}, from now on $\A=(L,E,R,\partial _{1},\partial _{2})$, of commutative algebras, is given by a chain complex  of commutative algebras,
	together with  actions $\blacktriangleright $ of $R$ on $E$ and $L$, such that $\d_1$ and $\d_2$  are module maps, where $R$ acts on itself by conjugation.  We also have an $R$-bilinear function (the Peiffer lifting):
	\begin{equation*}
	\{\ \ \otimes \ \ \}:E\otimes _{R}E\longrightarrow L,
	\end{equation*} satisfying the following axioms, for all $l,l^{\prime }\in L,\ e,e^{\prime },e^{\prime \prime
	}\in E,$ and $r\in R$:
	\begin{description}
		\item[2XM1)] $\partial _{2}\{e\otimes e^{\prime }\}=ee^{\prime }-\partial
		_{1}(e^{\prime })\blacktriangleright e$
		
		\item[2XM2)] $\{\partial _{2}(l)\otimes \partial _{2}(l^{\prime
		})\}=l l'$
		
		\item[2XM3)] $\{e\otimes e^{\prime }e^{\prime \prime }\}=\{ee^{\prime
		}\otimes e^{\prime \prime }\}+\partial _{1}(e^{\prime \prime
	})\blacktriangleright \{e\otimes e^{\prime }\}$
	
	\item[2XM4)] $\{\partial _{2}(l)\otimes e\}=e\blacktriangleright ^{\prime
	}l-\partial _{1}(e)\blacktriangleright l$
	
	\item[2XM5)] $\{e\otimes \partial _{2}(l)\}=e\blacktriangleright ^{\prime }l$
	
	\item[2XM6)] $r\blacktriangleright \{e\otimes e^{\prime
	}\}=\{r\blacktriangleright e\otimes e^{\prime }\}=\{e\otimes
	r\blacktriangleright e^{\prime }\}$
\end{description}

\end{definition}

\begin{remark}
	Note that {$\d_2\colon L \to E$} is a crossed
	module, where $E$ acts on $L$ as:%
	\begin{equation}\label{tprime}
	e\blacktriangleright ^{\prime }l=\{e\otimes \partial _{2}(l)\}
	\end{equation}%
	However {$\d_1\colon L \to E$} is in general only a
	pre-crossed module. The  Peiffer lifting in  $E$ measures exactly the failure of $\d_1\colon E \to R$ to be a crossed module.
\end{remark}

\begin{example}
	Let $(E,R,\partial )$ be a pre-crossed module. Then
	\smash{$\ker(\partial )\overset{i}{\longrightarrow }E\overset{\partial }{%
			\longrightarrow }R$}, where $i\colon \ker (\d) \to E$ is the inclusion map,
	is a 2-crossed module, where:
	\begin{equation*}
	\{\ \ \otimes \ \ \} \colon (e,e') \in  E\otimes _{R}E\longmapsto \{e\otimes e^{\prime }\}\doteq ee^\prime-\partial(e) \blacktriangleright e^\prime\in \ker(\partial).
	\end{equation*}
\end{example}

\begin{definition}[Freeness up to order one]\label{freeness} Let $\A=(L,E,R,\partial _{1},\partial _{2})$ be a
	2-crossed module. We say that $\A$ is free up to order
	one if $R$ is a free $\kappa$-algebra. In this paper, free up to order one 2-crossed modules will always come equipped with a specified chosen basis $B$ of $R$. Therefore $R$ will be the algebra $\kappa[B]$ of polynomials over $\kappa$, with a formal variable assigned to each element of $B$.
\end{definition}

Given 2-crossed modules $\A=(L,E,R,\partial _{1},\partial _{2})$ and $\A'=(L^{\prime },E^{\prime
},R^{\prime },\partial _{1}^{\prime },\partial _{2}^{\prime })$, a 2-crossed module map $f=(f_{2},f_{1},f_{0})\colon \A \to \A'$ consists of  algebra maps $f_{0}\colon R\to R^{\prime }$,
$f_{1} \colon E\to E^{\prime }$ and $f_{2} \colon L\to L^{\prime}$,
making the diagram:
$$ \xymatrix@R=20pt@C=60pt{
	L
	\ar[r]^{\partial _{2}}
	\ar[d]^{f_{2}}
	& E
	\ar[r]^{\partial _{1}}
	\ar[d]^{f_{1}}
	& R
	\ar[d]^{f_{0}}
	\\
	L^\prime
	\ar[r]_{\partial _{2}^{\prime }}
	& E^\prime
	\ar[r]_{\partial _{1}^{\prime }}
	& R^\prime
}$$%
\noindent commutative and  preserving the actions of $R$ and $R'$ and Peiffer liftings:
\begin{align*}
f_{1}(r\blacktriangleright e)&=f_{0}(r)\blacktriangleright
f_{1}(e),  \textrm{ for all }  e \in E  \textrm{ and } r\in R,
\\
f_{2}(r\blacktriangleright l)&=f_{0}(r)\blacktriangleright
f_{2}(l), \textrm{ for all }  l \in L  \textrm{ and } r\in R,
\\
f_{2}\{e\otimes e^{\prime }\}&=\{f_{1}(e)\otimes f_{1}(e^{\prime
})\},  \textrm{ for all } e,e^\prime\in E.
\end{align*}

\subsection{The algebras of 0-, 1-, 2- and 3-simplices on a 2-crossed module}\label{simp}

\noindent We now follow \cite{IJK} closely, referring to it for full calculations. Fix a 2-crossed module $\A=(L,E,R,\partial _{1},\partial _{2}{)}$ of algebras. 
{To begin introducing our notation, let us put $\A_0=R$ and call it the ``algebra of $0$-simplices'' in $\A$. Analogously, the  algebra $\A_1\doteq R\ltimes
	_{\blacktriangleright }E$ is called the ``algebra of 1-simplices'' in $\A$.
	We express each element $(r,e)\in \A_1$ in the simplicial form, below:}
\begin{align*}
\xymatrix{\pmor{r}{e}}.
\end{align*}
Note that restricting to the vertices yields two algebra maps $\A_1 \to \A_0$:
\begin{align*}
&\big(\xymatrix{\pmor{r}{e}}\big) \mapsto r \, , && &\big(\xymatrix{\pmor{r}{e}}\big) \mapsto r+\d_1(e).
\end{align*}
«
There exists an action $\blacktriangleright _{\bullet
}$ of $(R\ltimes _{\blacktriangleright }E)$ on $(E\ltimes
_{\blacktriangleright ^{\prime }}L)$, such that:
\begin{align*}
(r,e)\blacktriangleright _{\bullet }(e^{\prime },l)=(ee^{\prime }+r\blacktriangleright e^{\prime },\partial
_{1}(e)\blacktriangleright l+r\blacktriangleright l-\{e^{\prime }\otimes
e\}),
\end{align*}
for all $l\in L$, $e,e^{\prime }\in E$ and $r\in R$.
Using the action $\blacktriangleright _{\bullet }$, we construct the following commutative algebra:
\begin{equation*}
\A_2\doteq(R\ltimes _{\blacktriangleright}E)\ltimes _{\blacktriangleright
	_{\bullet }}(E\ltimes _{\blacktriangleright^{\prime } }L),
\end{equation*}
the ``algebra of 2-simplices'' in $\A$. We represent elements $(r,e,e^{\prime },l)\in A_{2}$ as:
\begin{equation}\label{2morsimp}
\xymatrix@R=15pt@C=15pt{\trinearlyn{l}{e^\prime}{e}{r}}
\end{equation}
We will not need the explicit form of the product in the algebra of 2-simplices in $\A$. Note however the following essential fact: restricting to the boundary of the triangle \eqref{2morsimp} yields three algebras maps 
$ \A_2 \to \A_1$.

In order to define the algebra of 3-simplices, additional actions are required:
First of all there exists an action $\blacktriangleright _{\ast }$ of $E\ltimes L$ on $L$,
with:
\begin{equation*}
(e,l)\blacktriangleright _{\ast }l^{\prime }=e\blacktriangleright ^{\prime
}l^{\prime }+ll^{\prime }, \quad \textrm{ where } l,l' \in L \textrm{ and } e \in E.
\end{equation*}
%  Using the action $\blacktriangleright _{\ast }$ we can construct the
% semidirect product $(E\ltimes _{\blacktriangleright ^{\prime }}L)\ltimes _{\blacktriangleright
% 	_{\ast }}L$.
We have
actions of $E$ and $R$ on $\left( (E\ltimes _{\blacktriangleright
	^{\prime }}L)\ltimes _{\blacktriangleright _{\ast }}L\right) $, with the form:
\begin{align*}
e\blacktriangleright^{1}_{e} (e^{\prime },l,l^{\prime })&=(ee^{\prime },\partial
_{1}(e)\blacktriangleright l-\{e^{\prime }\otimes e\},\partial
_{1}(e)\blacktriangleright l^{\prime }),\\
r\blacktriangleright^{1}_{r} (e^{\prime },l,l^{\prime })&=(r\blacktriangleright
e^{\prime },r\blacktriangleright l,r\blacktriangleright l^{\prime }).
\end{align*}
\noindent In particular, we have an action $\blacktriangleright ^{1}$ of $(R\ltimes
_{\blacktriangleright }E)$ on $\left( (E\ltimes _{\blacktriangleright
	^{\prime }}L)\ltimes _{\blacktriangleright _{\ast }}L\right) $, with:
\begin{equation*}
(r,e)\blacktriangleright ^{1}(e^{\prime },l,l^{\prime
})=(r\blacktriangleright e^{\prime }+ee^{\prime },r\blacktriangleright
l+\partial _{1}(e)\blacktriangleright l-\{e^{\prime }\otimes
e\},r\blacktriangleright l^{\prime }+\partial _{1}(e)\blacktriangleright
l^{\prime }).
\end{equation*}
\noindent We have actions of $E$ and $L$ on $\left( (E\ltimes _{\blacktriangleright
	^{\prime }}L)\ltimes _{\blacktriangleright _{\ast }}L\right) $, namely:
\begin{align*}
e\blacktriangleright^{2}_{e} (e',l,l^{\prime })&=(ee',e\blacktriangleright' l,\partial
_{1}(e)\blacktriangleright l^{\prime }-\{\partial _{2}(l)+e'\otimes e\}),\\
k\blacktriangleright^{2}_{l} (e^{\prime },l,l^{\prime })&=(0,e^{\prime
}\blacktriangleright' k+kl,-\{\partial _{2}(l)+e^{\prime }\otimes \partial
_{2}(k)\}).
\end{align*}
\noindent It follows that there exists an action $\blacktriangleright ^{2}$ of $(E\ltimes
_{\blacktriangleright ^{\prime }}L)$ on $\left( (E\ltimes
_{\blacktriangleright ^{\prime }}L)\ltimes _{\blacktriangleright _{\ast
	}}L\right) $:
	\begin{equation*}
	(e,l^{\prime \prime })\blacktriangleright ^{2}(e^{\prime },l,l^{\prime
	})=(ee^{\prime },e\blacktriangleright' l+e^{\prime }\blacktriangleright'
	l^{\prime \prime }+l^{\prime \prime }l,\partial _{1}(e)\blacktriangleright
	l^{\prime }-\{\partial _{2}(l)+e^{\prime }\otimes \partial _{2}(l^{\prime
		\prime })+e\}).
	\end{equation*}
	
	Putting everything together, it follows that  that there exists an action $\blacktriangleright _{\dagger }$ of $\left( (R\ltimes _{\blacktriangleright }E)\ltimes _{\blacktriangleright _{\bullet
		}}(E\ltimes _{\blacktriangleright ^{\prime }}L)\right) $ on  $\left(
		(E\ltimes _{\blacktriangleright ^{\prime }}L)\ltimes _{\blacktriangleright
			_{\ast }}L\right) $, such that:
		\begin{align*}
		(0,0,e,l^{\prime \prime })\blacktriangleright _{\dagger
		}(e^{\prime },l,l^{\prime })&=(e,l^{\prime \prime
	})\blacktriangleright ^{2}(e^{\prime },l,l^{\prime }),\\
	(r,e,0,0)\blacktriangleright _{\dagger }(e^{\prime },l,l^{\prime
	})&=(r,e)\blacktriangleright ^{1}(e^{\prime },l,l^{\prime }).
	\end{align*}

	We can now define the ``algebra of 3-simplices'' in $\A$ as being:
	\begin{align*}
	\A_3\doteq\big( (R\ltimes _{\blacktriangleright }E)\ltimes _{\blacktriangleright
		_{\bullet }}(E\ltimes _{\blacktriangleright ^{\prime }}L)\big) \ltimes
	_{\blacktriangleright _{\dagger }}\big( (E\ltimes _{\blacktriangleright
		^{\prime }}L)\ltimes _{\blacktriangleright _{\ast }}L\big).
	\end{align*}
	We  express the elements $(r,e,e^{\prime },l,e^{\prime \prime },l^{\prime },l^{\prime \prime })\in \A_3$ in the following  form:
	\begin{equation}\label{tet1}
	\xymatrix@R=22pt@C=30pt{\tetrnn{r}{e}{e^{\prime}}{l}{e^{\prime\prime}}{l^{\prime}}{l^{\prime\prime}}}
	\end{equation}
	
	As in the case of the algebra of 2-simplices, restricting to each face of the tetrahedron yields four algebra maps   $ \A_3 \to \A_2$. The back face of \eqref{tet1} is:
	$$ \xymatrix@R=10pt@C=15pt{\trianglyyy{r}{r+\partial_{1}(e)}{r+\partial_{1}(e)+\partial_{1}(e^{\prime})+\partial_{1}(e^{\prime\prime})}{e}{\,\,\,	\, e^{\prime}+e^{\prime\prime}+\partial_{2}(l)+\partial_{2}(l^{\prime})}{e+e^{\prime}+e^{\prime\prime}}{l+l^{\prime}}}$$\

	\subsection{Homotopy of 2-crossed module maps}\label{ph2cm}
	
	\noindent We continue to follow \cite{IJK}, where the reader can find the missing proofs. We now review the notion of homotopy between 2-crossed module maps.  Let us fix  2-crossed modules $\mathcal{A}%
	=(L,E,R,\partial _{1},\partial _{2}{)}$ and $\mathcal{A}^{\prime }=(L^{\prime
	},E^{\prime },R^{\prime },\partial _{1}^{\prime },\partial _{2}^{\prime }{)}$.	
	\begin{definition}\label{qder}
		Let $f=\big(f_2 ,f_1,f_0 )\colon \mathcal{A\to A}^{\prime }$ be a 2-crossed module
		map. A quadratic $f$-derivation is a pair
		$(s,t)$,
		where $s\colon R \to E'$, $t\colon E \to L'$ are linear maps, satisfying the following, for all $r,r' \in R$,  $e,e' \in E$:
		\begin{align}\label{der1}
		s(rr^{\prime })&= f_{0}(r)\blacktriangleright s(r^{\prime})+f_{0}(r^{\prime })\blacktriangleright s(r)+s(r)s(r^{\prime }),\\ \nonumber
		t(ee^{\prime })  &=  \{(s\circ \partial _{1})(e)\otimes
		f_{1}(e^{\prime })\}+\{(s\circ \partial _{1})(e^{\prime })\otimes
		f_{1}(e)\}+f_{1}(e)\blacktriangleright ^{\prime }t(e^{\prime
		})\\\nonumber &\quad \quad \quad +f_{1}(e^{\prime })\blacktriangleright ^{\prime }t(e) +(s\circ \partial _{1})(e)\blacktriangleright ^{\prime }t(e^{\prime
	}) \\\nonumber & \quad \quad \quad \quad \quad \quad +(s\circ \partial _{1})(e^{\prime })\blacktriangleright ^{\prime
}t(e)+t(e)t(e^{\prime }),\\ \nonumber
t(r\blacktriangleright e)&= f_{0}(r)\blacktriangleright
t(e)+(\partial _{1}^{\prime }\circ s)(r)\blacktriangleright
t(e)+\{s(r)\otimes f_{1}(e)\}\\\nonumber &\quad\quad \quad-\{f_{1}(e)\otimes s(r)\}-\{(s\circ \partial
_{1})(e)\otimes s(r)\}.
\end{align}
\end{definition}
\begin{remark}A map $s\colon R \to E'$ satisfying \eqref{der1} will be called an $f_{0}$-derivation.\end{remark}

\begin{theorem}	[Homotopy of 2-crossed module maps]\label{ph2cmm} 
	Let $f=(f_2,f_1,f_0)$ be a 2-crossed
	module map $\mathcal{A\to A}^{\prime }$. If $(s,t)$ is a quadratic $f$%
	-derivation, and if we define $g=(g_{2},g_{1},g_{0})$ as (for all $r\in R,\ e\in E$ and $l\in L$):
	\begin{equation*}
	\begin{split}
	g_{0}(r) & = f_{0}(r)+(\partial _{1}^{\prime } \circ s)(r) ,\\
	g_{1}(e) & = f_{1}(e)+(s \circ \partial _{1})(e)+(\partial _{2}^{\prime} \circ t)(e) ,\\
	g_{2}(l) & = f_{2}(l)+(t \circ \partial _{2})(l),
	\end{split}
	\end{equation*}
	then $g$ is a 2-crossed
	module map  $\mathcal{A\to A}^{\prime }$ and we use the notation:
	\smash{$f\ra{(f,s,t)}g
		$}. We
	say that $(f,s,t)$ is a homotopy (or quadratic derivation), from $f$ to $g$.
\end{theorem}

{Homotopies of 2-crossed modules maps $\A \to \A'$ normally cannot be composed, unless additional conditions are imposed on $\A$; see \cite{IJK,GM1}.}
{From now on we impose that $\A$ is free up to order one, with a chosen  basis $B$ of $R$. Thus $R$ is a  polynomial algebra, with a formal variable assigned to each element of $B$; Definition \ref{freeness}. 
	The reason to work in the free up to order one case is Lemma \ref{uniquederivation}, below. This in turn follows from the following, which  has an immediate proof:
	\begin{lemma}\label{defphi}{
			If $f=(f_2,f_1,f_0)\colon \A \to \A'$ is a 2-crossed modules map, then $f_{0}$-derivations are in one-to-one correspondence with algebra maps, of the form:}
		\begin{equation*}
		\phi : r \in  R \longmapsto \big(f_{0}(r),s(r)\big) \in R^{\prime }\ltimes _{\blacktriangleright
		}E^{\prime }\doteq \A_1', \textrm{ where } r\in R.
		\end{equation*}
	\end{lemma}
	\begin{lemma}\label{uniquederivation}
		In particular, if $R$ is a free $\k$-algebra, over the set $B$, an $f_{0}$-derivation $%
		s:R\to E^{\prime }$ can be uniquely specified by its value
		on $B \subset R$. Thus a set map $s^{\star}\colon B\to
		E^{\prime }$ uniquely extends to an $f_{0}$-derivation $s$; c.f.  below:
		\begin{equation*}
		\xymatrix@R=30pt@C=30pt{\free{B}{R}{\A_1}{(f_{0},s^{\star})}{(f_{0},s)}}
		\end{equation*}
	\end{lemma}
	
	Consider 2-crossed module maps $f,g,h\colon \A \to \A'$, with $\A$ free up to order one, with a chosen basis. Put  $f=(f_2,f_1,f_0)$, $g=(g_2,g_1,g_0)$ and $h=(h_2,h_1,h_0)$. Let $(s,t)$ be a quadratic $f$-derivation connecting $f$
	to $g$, and $(s^{\prime },t^{\prime })$ be a quadratic $g$-derivation connecting $g$
	to $h$. In order to set the scene for the similar case of 2-fold homotopies, to be dealt with later, let us recall how to define the concatenation $(s,t)\boxplus (s',t')=(s \boxplus s', t \boxplus t')$, connecting $f$ to $h$.
	
	We first let 
	$s\boxplus s^{\prime }\colon R \to E'$ be the unique $f_{0}$-derivation extending the restriction of  $s+s^{\prime } \colon R \to E$  to $B$. Thus $s \boxplus s'=s + s'$ in $B$, but in general not outside $B$. 
	By  definition of free algebra there exists a unique algebra map:
	\begin{equation*}
	X^{(s,s^{\prime })} \colon R \to \A'_2
	\end{equation*}%
	of the form below, for some (uniquely defined) linear map $w^{(s,s')}\colon R \to L'$:
	\begin{align*}
	{\xymatrix@C=4pt{ \\ X^{(s,s^{\prime })}(r)\quad =\quad }}\xymatrix@R=15pt@C=50pt{\trinearlyy{w^{(s,s')}(r)}{f_{0}(r)}{g_{0}(r)}{h_{0}(r)}{s(r)}{s^{\prime}(r)}{(s\boxplus s^\prime)(r)}}
	\end{align*}
	{Explicitly, $X^{(s,s')}$ is the unique extension (to an algebra map $R \to \A'_2$) of:}
	$$\xymatrix{\\  r \in B \mapsto } \xymatrix@R=15pt@C=50pt{\trinearlyy{0}{f_{0}(r)}{g_{0}(r)}{h_{0}(r)}{s(r)}{s^{\prime}(r)}{(s + s^\prime)(r)}}.$$
	Note that $w^{(s,s')}$ measures the distance from $(s\boxplus s^{\prime })(r)$ to $(s+s^{\prime })(r)$:
	\begin{equation}\label{dis}
	(s\boxplus s^{\prime })(r)=s(r)+s^{\prime }(r)-(\partial _{2}^{\prime }\circ w^{(s,s')})(r),\quad \quad  \textrm {for all } r \in R.
	\end{equation}
	{Clearly $w^{(s,s')}(r)$ vanishes for   $r\in B$. Also, 
		If $s=0$ or $s^\prime=0$ then 
		$ w^{(s,s^{\prime })}=0
		$.}

	By using the multiplication rules in the triangle space $\A'_2$ of $\A'$ (see \cite{IJK}) we can prove that  the function $w^{(s,s^{\prime })} \colon R\to L'$ is such that (for all $r,r' \in R$):}
\begin{multline*}
w^{(s,s')}(rr^{\prime })=f_{0}(r)\blacktriangleright w^{(s,s')}(r^{\prime
})+f_{0}(r^{\prime })\blacktriangleright w^{(s,s')}(r)+s^{\prime }(r^{\prime
})\blacktriangleright' w^{(s,s')}(r)
\\
+s(r^{\prime })\blacktriangleright' w^{(s,s')}(r)+s^{\prime
}(r)\blacktriangleright' w^{(s,s')}(r^{\prime })+s(r)\blacktriangleright' w^{(s,s')}(r^{\prime })
\\
-\{s^{\prime }(r^{\prime })\otimes s(r)\}-\{s^{\prime }(r)\otimes
s(r^{\prime })\}-w^{(s,s')}(r)w^{(s,s')}(r^{\prime }).
\end{multline*}
\begin{theorem}[Concatenation of  homotopies]\label{conchom} {Let $f,g,h \colon \A \to \A'$ be 2-crossed module maps, with $\A$ free up to order one, with a chosen basis. Consider homotopies $f \ra{(f,s,t)} g$ and $g \ra{(g,s',t')} h$. Put:}
	\begin{equation*}
	(t\boxplus t^{\prime })(e)=t(e)+t^{\prime }(e)+(w^{(s,s')} \circ \partial _{1})(e).
	\end{equation*} 
	Then the pair $(s\boxplus s^{\prime },t\boxplus t^{\prime })$ is a quadratic $f$-derivation connecting  $f$ to $h$.
\end{theorem}

This concatenation of homotopies is associative. Consider 2-crossed module maps $g,h,k\colon \A \to\A' $, with $\A$ is free up to order one. Let: 
\begin{itemize}
	\item $(s,t)$ be a quadratic $f$-derivation connecting $f$ to $g$: that is $f \ra{(f,s,t)} g$,
	\item $(s^{\prime },t^{\prime })$ be a quadratic $g$-derivation connecting $g$ to $h$: that is $g \ra{(g,s',t')} h$,
	\item $(s^{\prime\prime },t^{\prime\prime })$ be a quadratic $h$-derivation connecting $h$ to $k$: that is $h \ra{(h,s'',t'')} k$.
\end{itemize}
Then it holds (a consequence of the form of the multiplication in $\A_3'$; see \cite{IJK}):
\begin{equation}\label{compatibility}
w^{(s,s^{\prime })}(r)+w^{(s\boxplus
	s^{\prime },s^{\prime \prime })}(r)=w^{(s,s^{\prime }\boxplus s^{\prime \prime })}(r)+w^{(s^{\prime },s^{\prime \prime })}(r).
\end{equation}
From which it follows
$s\boxplus (s^{\prime }\boxplus s^{\prime \prime})=(s\boxplus s^{\prime })\boxplus s^{\prime \prime } $ and $t\boxplus (t^{\prime }\boxplus t^{\prime \prime})=(t\boxplus t^{\prime })\boxplus t^{\prime \prime }.$

The groupoid inverse of homotopies $\A \to \A'$, with $\A$ free up to order one, is also dealt with in \cite{IJK}. If 
$(s,t)$ is a quadratic $f$-derivation connecting $f$ to $g$. 
We can define a quadratic $g$-derivation $(\overline{s},\overline{t})$, connecting $g$ to $f$ in the following way:
Let $\bar{s}\colon R \to E$ be the unique $g_{0}$-derivation (Lemma \ref{uniquederivation}) extending  the restriction of the  function $-s$ to $B\subset R$. And also put  $\bar{t}=-t-(w^{(s,\bar{s})} \circ \partial _{1})$. Then $(\bar{s},\bar{t})$ is a quadratic $g$-derivation connecting $g$ to $f$, and is a left and right inverse of $(s,t)$ with respect to the operation of concatenation of homotopies.

Putting everything together:
\begin{theorem}\label{1group}
	Let $\A$ and $\A'$ be 2-crossed modules, of commutative algebras. {Suppose that  $\A=(L,E,R,\partial _{1},\partial _{2}{)}$ is free up to order one, with a chosen free basis $B$ of $R$.} We have a groupoid ${\rm HOM}(\A,\A')$, whose objects are the 2-crossed module maps $\A \to \A'$,  the morphisms being the homotopies between them.
\end{theorem}
\begin{remark} Let $f\colon \mathcal{A\to A}%
	^{\prime }$ be a 2-crossed module map.  Then the pair $(0_{s},0_{t})
	$, where each component is zero map, is  a quadratic $f$-derivation connecting $f$ to $f$. These are the identity morphisms in   ${\rm HOM}(\A,\A')$.
\end{remark}

\section{Quadratic 2-derivations and 2-fold homotopy}
\noindent Throughout this section, we fix  2-crossed modules $\mathcal{A}=(L,E,R,\partial _{1},\partial _{2})$ and $\mathcal{A}^{\prime }=(L^{\prime},E^{\prime },R^{\prime },\partial _{1}^{\prime },\partial _{2}^{\prime })$. {We do not suppose $\A$ to be free up to order one.}
\begin{definition}[Quadratic 2-derivation]\label{2quadder}
	Let $f \colon \mathcal{A\to A}^{\prime }$ be a 2-crossed module map and $(s,t)$ be a quadratic $f$-derivation. A quadratic $(f,s,t)$-2-derivation $q \colon R \to L'$ is a $\k$-linear function satisfying, for all $r,r^{\prime }\in R$:
	\begin{align*}
	q(rr^{\prime })=f_{0}(r)\blacktriangleright q(r^{\prime
	})+f_{0}(r^{\prime })\blacktriangleright q(r)+s(r)\blacktriangleright
	^{\prime }q(r^{\prime })+s(r^{\prime })\blacktriangleright ^{\prime
	}q(r)+q(r)q(r^{\prime }).
	\end{align*}
\end{definition}

\begin{theorem}[2-Fold Homotopy]\label{2foldhom}
	Let $(s,t)$ be a quadratic $f$-derivation connecting the 2-crossed module maps $f,g\colon \A \to \A'$.
	If $q$ is a quadratic $(f,s,t)$-2-derivation, and 
	if we define $(s^{\prime },t^{\prime })$ as (for all $e\in E$ and $r\in R$):
	\begin{equation}
	\begin{split}\label{source_target}
	s^{\prime }(r) &= s(r)+(\partial _{2}^{\prime }\circ q)(r) ,  \\	
	t^{\prime }(e) &= t(e)-(q\circ \partial _{1})(e) ,
	\end{split}
	\end{equation}
	then  $(s^{\prime },t^{\prime })$
	is  a quadratic $f$-derivation $f \ra{(f,s',t')} g$.
	We use the notation:
	\begin{align*}
	(f,s,t) \Ra{(f,s,t,q)} (f,s',t')
	\end{align*}%
	and say that the quadruple $(f,s,t,q)$ is a 2-fold homotopy (or quadratic 2-derivation) connecting $(f,s,t)$ to $(f,s',t')$. This all can be pictured as:
	$$
	\xymatrix@R=30pt@C=20pt{ f
		\ar@/^1.5pc/[rrrr]^{(f,s,t)}="1"
		\ar@/_1.5pc/[rrrr]_{(f,s^{\prime},t^{\prime})}="2"
		&&&& g
		\ar@{}"1";"2"|(.2){\,}="7"
		\ar@{}"1";"2"|(.8){\,}="8"
		\ar@{=>}"7" ;"8"|{(f,s,t,q)}
	}
	$$
\end{theorem}
\begin{proof}
	To make the formulae more compact, in this proof (and others), we do not use the ``$\circ$'' to denote composition, and put $\{e,f\}$, instead of  $\{e \otimes f\}$, for Peiffer liftings. Let us first see that the conditions of Definition \ref{qder} are satisfied. For all $e,e' \in E$ and $r,r^{\prime }\in R$, we have:
	\begin{align*}
	s^{\prime }&(rr^{\prime }) = s(rr^{\prime })+(\partial^{\prime}_{2} q)(rr^{\prime }) \\
	& = f_{0} (r)\blacktriangleright s(r^{\prime })+f_{0} (r^{\prime
	})\blacktriangleright s(r)+s(r)s(r^{\prime })+\partial^{\prime}_{2}[q(rr^{\prime			
})] \\
& = f_{0} (r)\blacktriangleright s(r^{\prime })+f_{0} (r^{\prime
})\blacktriangleright s(r)+s(r)s(r^{\prime })+\partial^{\prime}_{2}[f_{0}
(r)\blacktriangleright q(r^{\prime })\\
& \quad \quad +f_{0}(r^{\prime })\blacktriangleright q(r)+s(r)\blacktriangleright q(r^{\prime
})+s(r^{\prime })\blacktriangleright q(r)+q(r)q(r^{\prime })] \\
& = f_{0} (r)\blacktriangleright s(r^{\prime })+f_{0} (r^{\prime
})\blacktriangleright s(r)+s(r)s(r^{\prime })+\partial^{\prime}_{2}(f_{0}
(r)\blacktriangleright q(r^{\prime })) \\
& \quad \quad +\partial^{\prime}_{2}(f_{0} (r^{\prime
})\blacktriangleright q(r))+\partial^{\prime}_{2}(s(r)\blacktriangleright q(r^{\prime }))+\partial^{\prime}_{2}
(s(r^{\prime })\blacktriangleright q(r))+\partial^{\prime}_{2}(q(r)q(r^{\prime })) \\
& = f_{0} (r)\blacktriangleright s(r^{\prime })+f_{0} (r^{\prime
})\blacktriangleright s(r)+s(r)s(r^{\prime })+f_{0} (r)\blacktriangleright
(\partial^{\prime}_{2}q)(r^{\prime }) \\
& \quad \quad +f_{0} (r^{\prime })\blacktriangleright
(\partial^{\prime}_{2}q)(r)+s(r)(\partial^{\prime}_{2}q)(r^{\prime })+s(r^{\prime })
(\partial^{\prime}_{2}q)(r)+(\partial^{\prime}_{2}q)(r)(\partial^{\prime}_{2}q)(r^{\prime }) \\
& = f_{0} (r)\blacktriangleright \lbrack s(r^{\prime })+(\partial^{\prime}_{2} q)(r^{\prime })]+f_{0} (r^{\prime })\blacktriangleright
\lbrack s(r)+(\partial^{\prime}_{2}q)(r)] \\
& \quad \quad+[s(r)+(\partial^{\prime}_{2}q)(r)][s(r^{\prime })+(\partial^{\prime}_{2}q)
(r^{\prime })] \\
& = f_{0} (r)\blacktriangleright s^{\prime }(r^{\prime })+f_{0} (r^{\prime
})\blacktriangleright s^{\prime }(r)+s^{\prime }(r)s^{\prime }(r^{\prime }).
\end{align*}
Thus $s'$ is an $f_0$-derivation. We used the fact that: $\partial_{2}^{\prime}(r\blacktriangleright l)=r\blacktriangleright \partial_{2}^{\prime}(l)$.  Also:
\begin{align*}
t^{\prime }&(ee^{\prime })  = t(ee^{\prime })-(q\partial_{1} )(ee^{\prime
}) = t(ee^{\prime })-q(\partial_{1} (e)\partial_{1} (e^{\prime })) \\
& = \{(s\partial_{1} )(e),f_{1} (e^{\prime })\}+\{(s\partial_{1} )(e^{\prime }),f_{1}
(e)\}+f_{1} (e)\blacktriangleright ^{\prime }t(e^{\prime })+f_{1} (e^{\prime
})\blacktriangleright ^{\prime }t(e) \\
& \quad +(s\partial_{1} )(e)\blacktriangleright ^{\prime
}t(e^{\prime })+(s\partial_{1} )(e^{\prime })\blacktriangleright ^{\prime
}t(e)+t(e)t(e^{\prime })-[f_{0} (\partial_{1} (e))\blacktriangleright q(\partial_{1} (e^{\prime })) \\
& \quad \quad +f_{0}(\partial_{1} (e^{\prime }))\blacktriangleright q(\partial_{1} (e))+s(\partial_{1}
(e))\blacktriangleright' q(\partial_{1} (e^{\prime }))+s(\partial_{1} (e^{\prime
}))\blacktriangleright' q(\partial_{1} (e)) \\
& \quad \quad \quad +q(\partial_{1} (e))q(\partial_{1} (e^{\prime }))] \\
& = \{s(\partial_{1} (e)),f_{1} (e^{\prime })\}+\{(\partial_{2} ^{\prime }q)(\partial_{1}
(e)),f_{1} (e^{\prime })\}+\{s(\partial_{1} (e^{\prime })),f_{1} (e)\} \\
& \quad +\{(\partial_{2}^{\prime }q)(\partial_{1} (e^{\prime })),f_{1} (e)\}+f_{1} (e)\blacktriangleright
^{\prime }t(e^{\prime })-f_{1} (e)\blacktriangleright ^{\prime }(q\partial_{1} )(e^{\prime }) \\
& \quad \quad +f_{1}(e^{\prime })\blacktriangleright ^{\prime }t(e)-f_{1} (e^{\prime
})\blacktriangleright ^{\prime }(q\partial_{1} )(e)+(\partial_{2} ^{\prime }q)(\partial_{1} (e))\blacktriangleright' t(e^{\prime
}) \\
& \quad \quad \quad +s(\partial_{1} (e))\blacktriangleright' t(e^{\prime })-s(\partial_{1} (e))\blacktriangleright' (q\partial_{1} )(e^{\prime })-(\partial_{2} ^{\prime }q)(\partial_{1} (e))\blacktriangleright' (q\partial_{1} )(e^{\prime
}) \\
&  \quad \quad \quad \quad +s(\partial_{1} (e^{\prime }))\blacktriangleright' t(e)-s(\partial_{1} (e^{\prime
}))\blacktriangleright' (q\partial_{1} )(e)+(\partial_{2} ^{\prime }q)(\partial_{1} (e^{\prime
}))\blacktriangleright' t(e) \\
& \quad \quad \quad \quad \quad -(\partial_{2} ^{\prime }q)(\partial_{1} (e^{\prime }))\blacktriangleright' (q\partial_{1}
)(e)+t(e)t(e^{\prime })-t(e)(q\partial_{1} )(e^{\prime }) \\
& \quad \quad \quad \quad \quad \quad -(q\partial_{1} )(e)t(e^{\prime})+(q\partial_{1} )(e)(q\partial_{1} )(e^{\prime }) \\
& = \{s(\partial_{1} (e))+(\partial_{2} ^{\prime }q)(\partial_{1} (e)),f_{1} (e^{\prime
})\}+\{s(\partial_{1} (e^{\prime }))+(\partial_{2} ^{\prime }q)(\partial_{1} (e^{\prime
})),f_{1} (e)\} \\
& \quad +f_{1} (e)\blacktriangleright ^{\prime }[t(e^{\prime })-(q\partial_{1}
)(e^{\prime })]+f_{1} (e^{\prime })\blacktriangleright ^{\prime }[t(e)-(q\partial_{1}
)(e)] \\
& \quad \quad +[s(\partial_{1} (e))+(\partial_{2} ^{\prime }q)(\partial_{1} (e))]\blacktriangleright'
\lbrack t(e^{\prime })-(q\partial_{1} )(e^{\prime })] \\
& \quad \quad \quad +[s(\partial_{1} (e^{\prime }))+(\partial_{2} ^{\prime }q)(\partial_{1} (e^{\prime
}))]\blacktriangleright' \lbrack t(e)-(q\partial_{1} )(e)] \\
& \quad \quad \quad \quad +[t(e)-(q\partial_{1})(e)][t(e^{\prime })-(q\partial_{1} )(e^{\prime })] \\
& = \{(s^{\prime }\partial_{1} )(e),f_{1} (e^{\prime })\}+\{(s^{\prime }\partial_{1}
)(e^{\prime }),f_{1} (e)\}+f_{1} (e)\blacktriangleright ^{\prime }t^{\prime
}(e^{\prime })+f_{1} (e^{\prime })\blacktriangleright ^{\prime }t^{\prime
}(e) \\
& \quad +(s^{\prime }\partial_{1} )(e)\blacktriangleright ^{\prime }t^{\prime
}(e^{\prime })+(s^{\prime }\partial_{1} )(e^{\prime })\blacktriangleright ^{\prime }t^{\prime
}(e)+t^{\prime }(e)t^{\prime }(e^{\prime }).
\end{align*}
We used the second Peiffer axiom for $(\partial_{2}^{\prime}\colon L' \to E',\blacktriangleright ^{\prime } )$ in the final step; see \eqref{tprime}. Finally, by using the axioms 2XM2, 2XM4, 2XM5 (Definition \ref{2xmodule}) we get:
\begin{align*}
& t^{\prime }(r\blacktriangleright e) = t(r\blacktriangleright
e)-(q\partial_{1} )(r\blacktriangleright e) \\
& = t(r\blacktriangleright e)-q(\partial_{1} (r\blacktriangleright e))
= t(r\blacktriangleright e)-q(r\partial_{1} (e)) \\
& = f_{0} (r)\blacktriangleright t(e)+(\partial_{1} ^{\prime
}s)(r)\blacktriangleright t(e)+\{s(r),f_{1} (e)\}-\{f_{1} (e),s(r)\}-\{(s\partial_{1}
)(e),s(r)\} \\
& \quad -f_{0} (r)\blacktriangleright q(\partial_{1} (e))+f_{0} (\partial_{1}
(e))\blacktriangleright q(r)+s(r)\blacktriangleright' q(\partial_{1} (e))+s(\partial_{1}
(e))\blacktriangleright' q(r) \\
& \quad \quad +q(r)q(\partial_{1} (e)) \\
& = f_{0} (r)\blacktriangleright t(e)-f_{0} (r)\blacktriangleright (q\partial_{1}
)(e)+(\partial_{1} ^{\prime }s)(r)\blacktriangleright t(e)-(\partial_{1} ^{\prime
}s)(r)\blacktriangleright (q\partial_{1} )(e) \\
& \quad +\partial_{1} ^{\prime }(\partial_{2} ^{\prime }q(r))\blacktriangleright t(e)-\partial_{1}
^{\prime }(\partial_{2} ^{\prime }q(r))\blacktriangleright (q\partial_{1} )(e)+\{s(r),f_{1}
(e)\}+f_{1} (e)\blacktriangleright' q(r) \\
& \quad \quad -(\partial_{1} f_{1} )(e)\blacktriangleright q(r)-\{f_{1} (e),s(r)\}-f_{1}
(e)\blacktriangleright' q(r)-\{(s\partial_{1} )(e),s(r)\} \\
& \quad \quad \quad -(s\partial_{1})(e)\blacktriangleright' q(r)-s(r)\blacktriangleright' 
(q\partial_{1} )(e)+(\partial_{1} ^{\prime}s)(r)\blacktriangleright (q\partial_{1} )(e)-(q\partial_{1} )(e)q(r) \\
& = f_{0} (r)\blacktriangleright t^{\prime }(e)+(\partial_{1} ^{\prime
}s^{\prime })(r)\blacktriangleright t^{\prime }(e)+\{s^{\prime }(r),f_{1}
(e)\} \\
& \quad \quad  \quad -\{f_{1} (e),s^{\prime }(r)\}-\{(s^{\prime }\partial_{1} )(e),s^{\prime }(r)\}.
\end{align*}
It is now an easy calculation to verify that we do have $ f \ra{(f,s',t')} g$. 
\end{proof}

\subsection{Some algebraic properties of  quadratic 2-derivations}
Recall the definition of the algebra of 1-simplices in $\A=(L,E,R,\partial _{1},\partial _{2}{)}$, namely
$\A_1 \doteq R\ltimes E$; see \S \ref{simp}.  Simple calculations prove that there exists an action of $\A_{1}$ on $L$, of the form:
\begin{align*}
(r,e)\vartriangleright l = r\blacktriangleright l + e\blacktriangleright' l 
\end{align*}
for all $r \in R, e \in E$ and $l \in L$. We can therefore build the  semidirect product: 
\begin{align*}
Q \doteq (R'\ltimes_{\blacktriangleright} E')\ltimes_{\vartriangleright} L'.
\end{align*}

\begin{lemma}\label{defpsi}
	{By looking at the definition of $Q$, given a 2-crossed module homotopy $f \ra{(f,s,t)} g$, where $f,g \colon \A \to \A'$, quadratic $(f,s,t)$-2-derivations are in one-to-one correspondence with  algebra maps,
		of the form (here $r \in R$):}
	\begin{equation*}
	\psi : r \in  R \longmapsto \big(f_{0}(r),s(r),q(r)\big) \in Q=(R'\ltimes_{\blacktriangleright} E')\ltimes_{\vartriangleright} L'. 
	\end{equation*}
\end{lemma}
Therefore:
\begin{lemma}\label{uniquelyextended}
	Suppose that $\A$ is free up to order one, and $B$ a chosen basis of $R$. A quadratic $(f,s,t)$-2-derivation $q \colon R \to L'$ can be uniquely specified
	by its value on $B\subset R$. By this we mean that a set map $q^{\star} \colon B \to L'$ uniquely extends to a quadratic $(f,s,t)$-2-derivation $q$, fitting into
	the diagram below:
	\begin{equation*}
	\xymatrix@R=30pt@C=30pt{\free{B}{R}{Q'}
		{(f_{0},s^{\star},q^{\star})}{(f_{0},s,q)}}
	\end{equation*}
\end{lemma}
The following formulation of quadratic-2-derivations will also be essential later.
\begin{lemma}\label{2quadmap} Let $\A_2'=(R^{\prime }\ltimes
	_{\blacktriangleright }E^{\prime })\ltimes _{\blacktriangleright _{\bullet
		}}(E^{\prime }\ltimes _{\blacktriangleright ^{\prime }}L^{\prime })$ be the algebra of 2-simplices in $\A'$. Let $f\colon \A \to \A'$ be a 2-crossed module map. Let $(s,t)$ be a quadratic $f$ derivation. Let $q$ be a quadratic $(f,s,t)$-2-derivation. 
		We have an algebra homomorphism $\Omega \colon R \to\ A'_{2}$,  which has the form below, for each $r \in R$:
		\begin{equation*}
		\xymatrix{\\
			r \stackrel{\Omega}{\quad \longmapsto \quad} \big(f_{0}(r),s(r),(\partial_{2}^{\prime}\circ q)(r),-q(r)\big)&=}
		\xymatrix@R=10pt@C=20pt{\xtrinearlyy{-q(r)}{f_{0}(r)}{g_{0}(r)}{g_{0}(r)}{s(r)}{0}{s^{\prime}(r)}}
		\end{equation*}
	\end{lemma}
	
	\begin{proof}
		It is clear that $\Omega(r+r')=\Omega(r)+\Omega(r')$ and
		$\Omega(kr)=k \, \Omega(r)$.
		Put: \begin{align*}
		(r,e,e^{\prime },l) \, (r_{2},e_{2},e_{2}^{\prime },l_{2})=(a,b,c,d).
		\end{align*}
		Therefore  we have:
		\begin{align*}
		a&=rr_{2}, \quad \quad 
		b=r\blacktriangleright e_{2}+r_{2}\blacktriangleright e+ee_{2}, \\
		c&=ee_{2}^{\prime }+r\blacktriangleright e_{2}^{\prime }+e_{2}e^{\prime
		}+r_{2}\blacktriangleright e^{\prime }+e^{\prime }e_{2}^{\prime }, \\
		d&=\partial _{1}(e)\blacktriangleright l_{2}+r\blacktriangleright
		l_{2}-\{e_{2}^{\prime } , e\}+\partial _{1}(e_{2})\blacktriangleright
		l+r_{2}\blacktriangleright l-\{e^{\prime } , e_{2}\} \\
		& \quad \quad \quad \quad+e^{\prime}\blacktriangleright' l_{2}+e_{2}^{\prime }\blacktriangleright' l+ll_{2}.
		\end{align*}
		Thus:
		\begin{align*}
		& \Omega (r) \, \Omega (r^{\prime }) = \big( f_{0}(r),s(r),(\partial
		_{2}^{\prime }\circ q)(r),-q(r) \big) \, \, \big(f_{0}(r^{\prime }),s(r^{\prime }),(\partial
		_{2}^{\prime }\circ q)(r^{\prime }),-q(r^{\prime })\big) 
		\\ 
		& = \Big( f_{0}(r)f_{0}(r^{\prime }),f_{0}(r)\blacktriangleright s(r^{\prime
		})+f_{0}(r^{\prime })\blacktriangleright s(r)+s(r)s(r^{\prime }), \\ 
		& s(r)(\partial _{2}^{\prime }\circ q)(r^{\prime
		})+f_{0}(r)\blacktriangleright (\partial _{2}^{\prime }\circ q)(r^{\prime
	})+s(r^{\prime })(\partial _{2}^{\prime }\circ q)(r) \\
	& \quad \quad +f_{0}(r^{\prime
	})\blacktriangleright (\partial _{2}^{\prime }\circ q)(r)+(\partial
	_{2}^{\prime }\circ q)(r)(\partial _{2}^{\prime }\circ q)(r^{\prime }), \\ 
	& \partial' _{1}(s(r))\blacktriangleright -q(r^{\prime
	})+f_{0}(r)\blacktriangleright -q(r^{\prime })-\{(\partial _{2}^{\prime
}\circ q)(r^{\prime }) , s(r)\}+\partial' _{1}(s(r'))\blacktriangleright
-q(r) \\
& \quad \quad +f_{0}(r^{\prime })\blacktriangleright -q(r)-\{(\partial _{2}^{\prime
}\circ q)(r) , s(r^{\prime })\}+(\partial _{2}^{\prime }\circ
q)(r)\blacktriangleright -q(r^{\prime }) \\
& \quad \quad +(\partial _{2}^{\prime }\circ
q)(r^{\prime })\blacktriangleright -q(r)+q(r)q(r^{\prime }) \Big) 
\\ 
& = \Big( f_{0}(rr^{\prime }),s(rr^{\prime }), \\
& \partial _{2}^{\prime }\left(
f_{0}(r)\blacktriangleright q(r^{\prime })+f_{0}(r^{\prime
})\blacktriangleright q(r)+s(r)\blacktriangleright ^{\prime }q(r^{\prime
})+s(r^{\prime })\blacktriangleright ^{\prime }q(r)+q(r)q(r^{\prime
})\right) , \\ 
& -\partial' _{1}(s(r))\blacktriangleright q(r^{\prime
})-f_{0}(r)\blacktriangleright q(r^{\prime })-s(r)\blacktriangleright'
q(r^{\prime })+\partial' _{1}(s(r))\blacktriangleright q(r^{\prime}) \\
& \quad \quad -\partial' _{1}(s(r^{\prime }))\blacktriangleright q(r)-f_{0}(r^{\prime
})\blacktriangleright q(r)-s(r^{\prime })\blacktriangleright' q(r)+\partial'
_{1}(s(r^{\prime }))\blacktriangleright q(r) \\
& \quad \quad \quad \quad  -q(r)q(r^{\prime })-q(r^{\prime})q(r)+q(r)q(r^{\prime }) \Big) \\ 
& = \Big(f_{0}(rr^{\prime }),s(rr^{\prime }),\partial _{2}^{\prime }\left(
q(rr^{\prime })\right), \\
& -f_{0}(r)\blacktriangleright q(r^{\prime
})-s(r)\blacktriangleright' q(r^{\prime })-f_{0}(r^{\prime
})\blacktriangleright q(r)-s(r^{\prime })\blacktriangleright'
q(r)-q(r)q(r^{\prime }) \Big) \\ 
%&  \\ 
& = (f_{0}(rr^{\prime }),s(rr^{\prime }),(\partial _{2}^{\prime }\circ
q)(rr^{\prime }),-q(rr^{\prime }))  = \Omega (rr^{\prime }).
\end{align*}
\end{proof}

\subsection{Vertical composition of 2-fold homotopies}

\noindent {We continue to fix two  2-crossed modules $\mathcal{A%
	}=(L,E,R,\partial _{1},\partial _{2})$ and $\mathcal{A}^{\prime }=(L^{\prime
},E^{\prime },R^{\prime },\partial _{1}^{\prime },\partial _{2}^{\prime })$. Recall that  no freeness restriction is imposed  on $\A$. }
\begin{theorem}
	Let $f,g: \A \to \A'$ are homotopic 2-crossed module morphisms connected by the quadratic derivations $(f,s,t)$, $(f,s^{\prime},t^{\prime})$ and $(f,s^{\prime\prime},t^{\prime\prime})$. Let:
	\begin{itemize}
		\item[-] $q$ be a  quadratic $(f,s,t)$-2-derivation with $(f,s,t) \Ra{(f,s,t,q)} (f,s',t')$,
		\item[-] $q'$ be a quadratic $(f,s',t')$-2-derivation with $(f,s',t') \Ra{(f,s',t',q')} (f,s'',t'')$.
	\end{itemize}
	Then the map $q\star q^{\prime}:R \rightarrow L^{\prime}$ such that:
	$	(q\star q^{\prime})(r)=q(r)+q^{\prime}(r)
	$
	is a quadratic $(f,s,t)$-2-derivation connecting $(f,s,t)$ to $(f,s^{\prime\prime},t^{\prime\prime})$. This defines the vertical composition of quadratic derivations that can be pictured as:
	
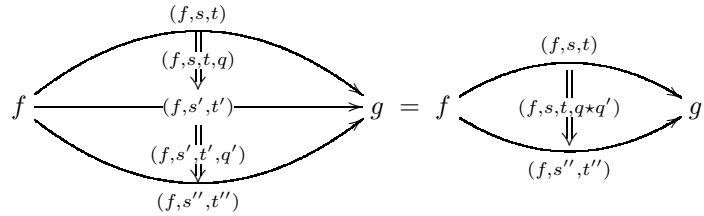
\begin{equation*}
	\xymatrix@R=10pt@C=20pt{ f \ar@/^2.4pc/[rrrrr]^{(f,s,t)}="1"
		\ar[rrrrr]|{(f,s^{\prime},t^{\prime})}="2"
		\ar@/_2.4pc/[rrrrr]_{(f,s^{\prime\prime},t^{\prime\prime})}="9"
		\ar@{}"1";"2"|(.2){\,}="7"
		\ar@{}"1";"2"|(.8){\,}="8"
		\ar@{=>}"7";"8"|{(f,s,t,q)}
		\ar@{}"2";"9"|(.2){\,}="10"
		\ar@{}"2";"9"|(.8){\,}="11"
		\ar@{=>}"10" ;"11"|{(f,s^{\prime},t^{\prime},q^{\prime})}
		&&&&& g }=
	\xymatrix{ f
		\ar@/^1.4pc/[rrr]^{(f,s,t)}="1"
		\ar@/_1.4pc/[rrr]_{(f,s^{\prime\prime},t^{\prime\prime})}="2"
		&&& g
		\ar@{}"1";"2"|(.2){\,}="7"
		\ar@{}"1";"2"|(.8){\,}="8"
		\ar@{=>}"7" ;"8"|{(f,s,t,q\star q^{\prime})}
	}
	\end{equation*}
\end{theorem}

\begin{proof} 
	This follows from the fact that $(L \to E,\blacktriangleright')$ is a crossed module:
	\begin{align*}
	(q&\star q^{\prime })(rr^{\prime }) = q(rr^{\prime })+q^{\prime
	}(rr^{\prime }) \\ 
	& = f_{0}(r)\blacktriangleright q(r^{\prime })+f_{0}(r^{\prime
	})\blacktriangleright q(r)+s(r)\blacktriangleright ^{\prime }q(r^{\prime
})+s(r^{\prime })\blacktriangleright ^{\prime }q(r)+q(r)q(r^{\prime }) \\ 
& \quad +f_{0}(r)\blacktriangleright q^{\prime }(r^{\prime })+f_{0}(r^{\prime
})\blacktriangleright q^{\prime }(r)+s^{\prime }(r)\blacktriangleright
^{\prime }q^{\prime }(r^{\prime })+s^{\prime }(r^{\prime
})\blacktriangleright ^{\prime }q^{\prime }(r) \\
& \quad \quad  +q^{\prime }(r)q^{\prime
}(r^{\prime }) \\ 
& = f_{0}(r)\blacktriangleright q(r^{\prime })+f_{0}(r^{\prime
})\blacktriangleright q(r)+s(r)\blacktriangleright ^{\prime }q(r^{\prime
})+s(r^{\prime })\blacktriangleright ^{\prime }q(r)+q(r)q(r^{\prime }) \\ 
& \quad +f_{0}(r)\blacktriangleright q^{\prime }(r^{\prime })+f_{0}(r^{\prime
})\blacktriangleright q^{\prime }(r)+[s(r)+(\partial _{2}^{\prime }\circ
q)(r)]\blacktriangleright ^{\prime }q^{\prime }(r^{\prime }) \\
& \quad \quad +[s(r)+(\partial_{2}^{\prime }\circ q)(r)]\blacktriangleright 
^{\prime }q^{\prime}(r)+q^{\prime }(r)q^{\prime }(r^{\prime }) \\ 
& = f_{0}(r)\blacktriangleright q(r^{\prime })+f_{0}(r^{\prime
})\blacktriangleright q(r)+s(r)\blacktriangleright ^{\prime }q(r^{\prime
})+s(r^{\prime })\blacktriangleright ^{\prime }q(r)+q(r)q(r^{\prime
}) \\ 
& \quad +f_{0}(r)\blacktriangleright q^{\prime }(r^{\prime })+f_{0}(r^{\prime
})\blacktriangleright q^{\prime }(r)+s(r)\blacktriangleright ^{\prime
}q^{\prime }(r^{\prime })+s(r^{\prime })\blacktriangleright ^{\prime
}q^{\prime }(r) \\ 
& \quad \quad + q^{\prime }(r)q^{\prime }(r^{\prime }) +q(r)q^{\prime }(r^{\prime })+q^{\prime}(r)q(r^{\prime }) 
\rm{\quad \quad (\because \text{   XM2})} \\
& = f_{0}(r)\blacktriangleright (q\star q^{\prime })(r^{\prime
})+f_{0}(r^{\prime })\blacktriangleright (q\star q^{\prime
})(r)+s(r)\blacktriangleright ^{\prime }(q\star q^{\prime })(r^{\prime
}) \\
& \quad +s(r^{\prime })\blacktriangleright ^{\prime }(q\star q^{\prime
})(r)+(q\star q^{\prime })(r)(q\star q^{\prime })(r^{\prime }),
\end{align*}
for all $r,r' \in R$. An easy calculation now proves that  $(f,s,t) \Ra{(f,s,t,q)} (f,s^{\prime\prime},t^{\prime\prime})$.
\end{proof}

The vertical composition of 2-fold homotopies is associative. Also note that the zero quadratic $(f,s,t)$-2-derivation defines an identity arrow for $(f,s,t)$. Therefore, given two 2-crossed module maps $f,g \colon \A \to \A'$, we have a category  ${\rm HOM}_{2}(\mathcal{A},\mathcal{A^{\prime}},f,g)$ with objects being the quadratic $f$-derivations $(s,t)$ with $f \ra{(f,s,t)} g$, the morphisms being the 2-fold homotopies connecting them. 

\medskip

\noindent By the following lemma, ${\rm HOM}_{2}(\mathcal{A},\mathcal{A^{\prime}},f,g)$ is a groupoid:
\begin{lemma}
	Given a quadratic $(f,s,t)$-2-derivation $q$ which connects $(f,s,t)$ to $(f,s',t')$, then the map $\bar{q}\colon R \to L'$ such that
	$\bar{q}(r)= -q(r)$ is a quadratic $(f,s^{\prime},t^{\prime})$-2-derivation connecting $(f,s^{\prime},t^{\prime})$ to $(f,s,t)$; clearly being the inverse of $q$, with respect to the vertical composition of 2-fold homotopies.
\end{lemma}

\begin{proof}
	$\overline{q}$ is a quadratic $(f,s',t')$-2-derivation, since (for all $r,r' \in R$):
	\begin{align*}
	&\bar{q}(rr^{\prime })  = -(q(rr^{\prime }))=-[ f_{0}(r)%
	\blacktriangleright q(r^{\prime })+f_{0}(r^{\prime })\blacktriangleright
	q(r)+s(r)\blacktriangleright ^{\prime }q(r^{\prime }) \\
	& \quad \quad +s(r^{\prime})\blacktriangleright ^{\prime }q(r)+q(r)q(r^{\prime })]  \\ 
	& = -f_{0}(r)\blacktriangleright q(r^{\prime })-f_{0}(r^{\prime
	})\blacktriangleright q(r)-s(r)\blacktriangleright ^{\prime }q(r^{\prime
})-s(r^{\prime })\blacktriangleright ^{\prime }q(r)-q(r)q(r^{\prime }) \\ 
& \quad \quad +q(r)q(r^{\prime })-q(r)q(r^{\prime }) \\ 
& = -f_{0}(r)\blacktriangleright q(r^{\prime })-f_{0}(r^{\prime
})\blacktriangleright q(r)-s(r)\blacktriangleright ^{\prime }q(r^{\prime
})-s(r^{\prime })\blacktriangleright ^{\prime }q(r)-q(r)q(r^{\prime }) \\ 
& \quad \quad +(\partial _{2}^{\prime }\circ q)(r)\blacktriangleright ^{\prime
}q(r^{\prime })-(\partial _{2}^{\prime }\circ q)(r)\blacktriangleright
^{\prime }q(r^{\prime }) \\ 
& = -f_{0}(r)\blacktriangleright q(r^{\prime })-f_{0}(r^{\prime
})\blacktriangleright q(r) \\ 
& \quad \quad -(s(r)+(\partial _{2}^{\prime }\circ q)(r))\blacktriangleright
^{\prime }q(r^{\prime })-(s(r^{\prime })+(\partial _{2}^{\prime }\circ
q)(r^{\prime }))\blacktriangleright ^{\prime }q(r)+q(r)q(r^{\prime }) \\ 
& = f_{0}(r)\blacktriangleright -q(r^{\prime })+f_{0}(r^{\prime
})\blacktriangleright -q(r)-s^{\prime }(r)\blacktriangleright ^{\prime
}q(r^{\prime })-s^{\prime }(r^{\prime })\blacktriangleright ^{\prime
}q(r)+q(r)q(r^{\prime }) \\  
& = f_{0}(r)\blacktriangleright \bar{q}(r^{\prime })+f_{0}(r^{\prime
})\blacktriangleright \bar{q}(r)+s^{\prime }(r)\blacktriangleright ^{\prime }%
\bar{q}(r^{\prime })+s^{\prime }(r^{\prime })\blacktriangleright ^{\prime }%
\bar{q}(r)+\bar{q}(r)\bar{q}(r^{\prime }).
\end{align*}
This again is a consequence of the second Peiffer relation of the crossed module $(\partial_{2}^{\prime}\colon L' \to E',\blacktriangleright ^{\prime })$. The other assertions are immediate. 
\end{proof}

\section{A 2-groupoid of  ${\rm HOM}(\A,\A')_2$ of 2-fold homotopies}

Fix two 2-crossed modules $\mathcal{A}=(L,E,R,\partial _{1},\partial _{2})$ and $\mathcal{A}^{\prime }=(L^{\prime},E^{\prime },R^{\prime },\partial _{1}^{\prime },\partial _{2}^{\prime })$. We will now need to take $\A$ to be free up to order one, and will consider a chosen free basis $B$ of the algebra  $R$.  
Recall the construction of the groupoid ${\rm HOM}(\A,\A')$ of 2-crossed module maps $\A \to \A'$, and their homotopies; \S \ref{ph2cm}.
\subsection{Right whiskering of 2-fold homotopies}

\noindent Let $f,g,h\colon \A \to \A'$ be 2-crossed module maps. Consider  homotopies:
\begin{align*}
&f \ra{(f,s,t)}g , && f\ra{(f,s^{\prime},t^{\prime})} g , && g\ra{(g,u,v)}h,
\end{align*}
and  a 2-fold homotopy $(f,s,t)\Ra{(f,s,t,q)}(f,s^{\prime},t^{\prime})$. This fits inside the diagram:
\begin{equation*}
\xymatrix{ f
	\ar@/^1.4pc/[rrr]^{(f,s,t)}="1"
	\ar@/_1.4pc/[rrr]_{(f,s^{\prime},t^{\prime})}="2"
	&&& g
	\ar@{}"1";"2"|(.2){\,}="7"
	\ar@{}"1";"2"|(.8){\,}="8"
	\ar@{=>}"7" ;"8"|{(f,s,t,q)}
	\ar[rr]^{(g,u,v)}
	&& h
}.
\end{equation*}
For defining the right whiskering of 2-fold homotopies by 2-crossed module (1-fold) homotopies, we must define a new quadratic 2-derivation $(q\circledcirc u)$ which connects $(f,s\boxplus u,t\boxplus v)$ to $(f,s^{\prime}\boxplus u,t^{\prime}\boxplus v)$, therefore fitting inside the diagram:
\begin{equation*}
\xymatrix{f
	\ar@/^1.3pc/[rrr]^{(f,s\boxplus u,t\boxplus v)}="1"
	\ar@/_1.3pc/[rrr]_{(f,s^{\prime}\boxplus u,t^{\prime}\boxplus v)}="2"
	&&& h
	\ar@{}"1";"2"|(.2){\,}="7"
	\ar@{}"1";"2"|(.8){\,}="8"
	\ar@{=>}"7" ;"8"|{\;  (f,s\boxplus u,t\boxplus v, q\circledcirc u)  }
}.
\end{equation*}

Let $(q\circledcirc u)$ be the unique quadratic $(f,s\boxplus u,t\boxplus v)$-2-derivation (by using Remark \ref{uniquelyextended}) extending the function $q_{|B}$ (the restriction of $q$ to $B$); see  the diagram below:
\begin{align*}
\xymatrix@R=35pt@C=35pt{\free{B}{R}{Q'}{(f_{0},s+u,q_{|B})}{\big (f_{0},s\boxplus u,q\circledcirc u\big)}}
\end{align*}
By definition, we have (for all $b \in B$):  $$(q\circledcirc u)(b)=q(b).$$

\medskip

Let us prove that $(q\circledcirc u)$ indeed connects $(f,s\boxplus u,t\boxplus v)$ to $(f,s^{\prime}\boxplus u,t^{\prime}\boxplus v)$. We must check the conditions given in equation \eqref{source_target}. 

\medskip

Consider the map:
\begin{equation*}
\zeta \colon b \in B \subset R\longmapsto \big(f_{0}(b),(s+u)(b),(\partial_{2}' \circ q)(b),-q(b)\big) \in  \A_2'.
\end{equation*}
Geometrically, for a $b\in B$, the free basis of $R$, the element $\zeta (b)$ has the form:
\begin{align*}
% \hskip -1.55cm
\xymatrix@R=20pt@C=7pt{\xxtrinearly{-q(b)}{(\partial_{2}' \circ q)(b)}{(s+u)(b)}{f_{0}(b)}} 
{\xymatrix{\\\quad=\quad \\}} \xymatrix@R=20pt@C=20pt{\xtrinearlyy{-q(b)}{f_{0}(b)}{h_{0}(b)}{h_{0}(b)}{(s+u)(b)\quad\,\,}{0}{(s^{\prime}+u)(b)}}
\end{align*}
Let $Y^{(s,s',u)}:R\to \A_2'$ be the unique algebra map 
extending $\zeta $. From lemma \ref{2quadmap} and the fact that restricting to each side of the triangle yields an algebra map, it follows that for each $r \in R$:
\begin{align*}
\xymatrix{\\ Y^{(s,s',u)}(r)= } \xymatrix@R=20pt@C=40pt{ \xtrinearlyy{-(q\circledcirc u)(r)}{f_{0}(r)}{h_{0}(r)}{h_{0}(r)}{(s \boxplus u)(r)}{0}{(s^{\prime}\boxplus u)(r)}}
\end{align*}
Therefore, the homomorphism $Y^{(s,s^{\prime})}\colon  R \to \A'_2$ is of the form:
\begin{align*}
Y^{(s,s',u)}(r)=\big(f_{0}(r),(s\boxplus u)(r),\varrho (r),-(q\circledcirc u)(r)\big),
\end{align*}
(recall Remark \ref{2quadmap}) for some map $\varrho\colon R \to E'$. In particular, we can see that $\varrho (r)-\partial _{2}((q\circledcirc u)(r))=0$, thus (for all $r \in R$):
\begin{align*}
\varrho (r)=\partial_{2}' \big((q\circledcirc u)(r)\big).
\end{align*}
In particular, by equation \eqref{2morsimp},  our homomorphism has the following form:
\begin{align*}
Y^{(s,s',u)}(r)=\big(f_{0}(r),(s\boxplus u)(r),\partial_{2}'((q\circledcirc u)(r)),-(q\circledcirc u)(r)\big).
\end{align*}
Finally, by using the Lemma \ref{2quadmap}, we get the conclusion:
\begin{align}\label{1ndcond}
(s^{\prime}\boxplus u)(r)=(s\boxplus u)(r)+(\partial _{2}^{\prime }\circ (q\circledcirc u))(r),
\end{align}
which gives us the first equality of \eqref{source_target}, in  Theorem \ref{2foldhom}.

\noindent To finish proving $(f,s\boxplus u,t\boxplus v) \Ra{\big(f,s\boxplus u,t\boxplus v,  q\circledcirc u\big ) }(f,s^{\prime}\boxplus u,t^{\prime}\boxplus v)$, we need: 
\begin{lemma}
	{In the condition above we have:}
	\begin{equation*}
	\partial_{2}' \big( (q\circledcirc u)(r) \big) = \partial_{2}' \big( q(r)+w^{(s,u)}(r)-w^{(s^{\prime},u)}(r) \big).
	\end{equation*}
\end{lemma}
\begin{proof}
	Combining
	$(s^{\prime}\boxplus u)(r)=(s\boxplus u)(r)+(\partial _{2}^{\prime }\circ (q\circledcirc u))(r)$ with equation \eqref{dis}:
	\begin{align*}
	(s\boxplus u)(r)+(\partial _{2}^{\prime }\circ (q\circledcirc u))(r) = s(r)+u(r)-(\partial
	_{2}^{\prime }\circ w^{(s,u)})(r)+(\partial _{2}^{\prime }\circ (q\circledcirc u))(r).
	\end{align*}%
	On the other hand we know:%
	\begin{align*}
	(s^{\prime }\boxplus u)(r) & = s^{\prime }(r)+u(r)-(\partial _{2}^{\prime }\circ w^{(s^{\prime },u)})(r) \\
	& = s(r)+(\partial _{2}^{\prime }\circ q)(r)+u(r)-(\partial _{2}^{\prime }\circ w^{(s^{\prime },u)})(r).
	\end{align*}
	This  gives us the stated formula.
\end{proof}

\noindent We now provide a more explicit formula for the quadratic 2-derivation $(q\circledcirc u)$.
\begin{theorem}\label{rwhisker}
	If $q$ is a quadratic $(f,s,t)$-2-derivation connecting the homotopy $(s,t)$ to $(s^{\prime},t^{\prime})$ and $(u,v)$ is a $g$-derivation, connecting $g$ to $h$, all fitting in:
	\begin{equation*}
	\xymatrix{ f
		\ar@/^1.4pc/[rrr]^{(f,s,t)}="1"
		\ar@/_1.4pc/[rrr]_{(f,s^{\prime},t^{\prime})}="2"
		&&& g
		\ar@{}"1";"2"|(.2){\,}="7"
		\ar@{}"1";"2"|(.8){\,}="8"
		\ar@{=>}"7" ;"8"|{(f,s,t,q)}
		\ar[rr]^{(g,u,v)}
		&& h
	}
	\end{equation*}
	Then the map $q\hat{\circledcirc} u\colon R \to L'$, defined as:
	\begin{align*}
	(q\hat{\circledcirc} u)(r)=q(r)+w^{(s,u)}(r)-w^{(s',u)}(r),
	\end{align*}
	defines a quadratic $(f,s\boxplus u,t\boxplus v)$-2-derivation. In particular, given that  $(q\circledcirc u)$ and $(q\hat{\circledcirc} u)$ coincide in $B$,  a free basis of $R$, it thus follow that, for each $r \in R$: 
	\begin{equation}\label{formulawisk}
	(q\circledcirc u)(r)=q(r)+w^{(s,u)}(r)-w^{(s',u)}(r).
	\end{equation}
\end{theorem}
\begin{proof}
	To prove this, we need to verify the condition (recall Definition \ref{2quadder}):
	\begin{align*}
	(q\hat{\circledcirc} u)(rr^{\prime })& = f_{0}(r)\blacktriangleright (q\hat{\circledcirc}
	u)(r^{\prime })+f_{0}(r^{\prime })\blacktriangleright (q\hat{\circledcirc} u)(r) \\
	& \quad \quad +(s\boxplus u)(r)\blacktriangleright ^{\prime }(q\circledcirc u)(r^{\prime
	})+(s\boxplus u)(r^{\prime })\blacktriangleright ^{\prime }(q\circledcirc
	u)(r) \\
	& \quad \quad \quad \quad +(q\hat{\circledcirc} u)(r)(q\hat{\circledcirc} u)(r^{\prime }).
	\end{align*}%
	If we expand the left-hand-side we get:
	\begin{align*}
	& (q\hat{\circledcirc} u)(rr^{\prime }) = q(rr^{\prime
	})+w^{(s,u)}(rr^{\prime })-w^{(s^{\prime },u)}(rr^{\prime }) \\ 
	& \quad = f_{0}(r)\blacktriangleright q(r^{\prime })+f_{0}(r^{\prime
	})\blacktriangleright q(r)+s(r)\blacktriangleright ^{\prime }q(r^{\prime
})+s(r^{\prime })\blacktriangleright ^{\prime }q(r)+q(r)q(r^{\prime }) \\ 
& \quad \quad +f_{0}(r)\blacktriangleright w^{(s,u)}(r^{\prime })+f_{0}(r^{\prime
})\blacktriangleright w^{(s,u)}(r)+u(r^{\prime })\blacktriangleright'
w^{(s,u)}(r) \\ 
& \quad \quad \quad +s(r^{\prime })\blacktriangleright'
w^{(s,u)}(r)+u(r)\blacktriangleright' w^{(s,u)}(r^{\prime
})+s(r)\blacktriangleright' w^{(s,u)}(r^{\prime }) \\ 
& \quad \quad \quad \quad -\{u(r^{\prime }) , s(r)\}-\{u(r) , s(r^{\prime
})\}-w^{(s,u)}(r)w^{(s,u)}(r^{\prime }) \\ 
& \quad \quad \quad \quad \quad -[f_{0}(r)\blacktriangleright w^{(s^{\prime },u)}(r^{\prime
})+f_{0}(r^{\prime })\blacktriangleright w^{(s^{\prime },u)}(r)+u(r^{\prime
})\blacktriangleright' w^{(s^{\prime },u)}(r) \\ 
& \quad \quad \quad \quad \quad \quad +s^{\prime }(r^{\prime })\blacktriangleright' w^{(s^{\prime
	},u)}(r)+u(r)\blacktriangleright' w^{(s^{\prime },u)}(r^{\prime })+s^{\prime
}(r)\blacktriangleright' w^{(s^{\prime },u)}(r^{\prime }) \\ 
& \quad \quad \quad \quad \quad \quad \quad -\{u(r^{\prime }) , s^{\prime }(r)\}-\{u(r) , s^{\prime
}(r^{\prime })\}-w^{(s^{\prime },u)}(r)w^{(s^{\prime },u)}(r^{\prime })].
\end{align*}
while on the right-hand-side we get:
\begin{align*}
& f_{0}(r)\blacktriangleright (q\hat{\circledcirc} u)(r^{\prime
})+f_{0}(r^{\prime })\blacktriangleright (q\circledcirc u)(r)+(s\boxplus
u)(r)\blacktriangleright ^{\prime }(q\hat{\circledcirc} u)(r^{\prime }) \\
& \quad \quad \quad \quad \quad \quad \quad \quad \quad \quad 
+(s\boxplus
u)(r^{\prime })\blacktriangleright ^{\prime }(q\hat{\circledcirc}
u)(r)+(q\hat{\circledcirc} u)(r)(q\hat{\circledcirc} u)(r^{\prime }) \\ 
& = f_{0}(r)\blacktriangleright \lbrack q(r^{\prime
})+w^{(s,u)}(r^{\prime })-w^{(s^{\prime },u)}(r^{\prime })]+f_{0}(r^{\prime
})\blacktriangleright \lbrack q(r)+w^{(s,u)}(r)-w^{(s^{\prime },u)}(r)] \\ 
& \quad +[s(r)+u(r)-(\partial _{2}^{\prime }\circ
w^{(s,u)})(r)]\blacktriangleright' \lbrack q(r^{\prime })+w^{(s,u)}(r^{\prime
})-w^{(s^{\prime },u)}(r^{\prime })] \\ 
& \quad \quad +[s(r^{\prime })+u(r^{\prime })-(\partial _{2}^{\prime }\circ
w^{(s,u)}(r^{\prime }))]\blacktriangleright' \lbrack
q(r)+w^{(s,u)}(r)-w^{(s^{\prime },u)}(r)] \\ 
& \quad \quad \quad +[q(r)+w^{(s,u)}(r)-w^{(s^{\prime },u)}(r)][q(r^{\prime
})+w^{(s,u)}(r^{\prime })-w^{(s^{\prime },u)}(r^{\prime })] \\ 
& = f_{0}(r)\blacktriangleright q(r^{\prime
})+f_{0}(r)\blacktriangleright w^{(s,u)}(r^{\prime
})-f_{0}(r)\blacktriangleright w^{(s^{\prime },u)}(r^{\prime
}) \\
& \quad 
+f_{0}(r^{\prime })\blacktriangleright q(r)+f_{0}(r^{\prime
})\blacktriangleright w^{(s,u)}(r)-f_{0}(r^{\prime })\blacktriangleright
w^{(s^{\prime },u)}(r) \\ 
& \quad \quad 
+s(r)q(r^{\prime })+s(r)w^{(s,u)}(r^{\prime })-s(r)w^{(s^{\prime
	},u)}(r^{\prime }) \\ 
& \quad \quad \quad 
+u(r)q(r^{\prime })+u(r)w^{(s,u)}(r^{\prime })-u(r)w^{(s^{\prime
	},u)}(r^{\prime }) \\ 
& \quad \quad \quad \quad
-(\partial _{2}^{\prime }\circ w^{(s,u)})(r)q(r^{\prime })-(\partial
_{2}^{\prime }\circ w^{(s,u)})(r)w^{(s,u)}(r^{\prime }) \\
& \quad
+(\partial_{2}^{\prime }\circ w^{(s,u)})(r)w^{(s^{\prime },u)}(r^{\prime }) \\ 
& \quad \quad
+s(r^{\prime })q(r)+s(r^{\prime })w^{(s,u)}(r)-s(r^{\prime
})w^{(s^{\prime },u)}(r) \\ 
& \quad \quad \quad 
+u(r^{\prime })q(r)+u(r^{\prime })w^{(s,u)}(r)-u(r^{\prime
})w^{(s^{\prime },u)}(r) \\ 
& \quad \quad \quad \quad 
-(\partial _{2}^{\prime }\circ w^{(s,u)})(r^{\prime })q(r)-(\partial
_{2}^{\prime }\circ w^{(s,u)})(r^{\prime })w^{(s,u)}(r) \\
& \quad 
+(\partial_{2}^{\prime }\circ w^{(s,u)})(r^{\prime })w^{(s^{\prime },u)}(r) \\ 
& \quad \quad 
+q(r)q(r^{\prime })+q(r^{\prime })w^{(s,u)}(r^{\prime
})-q(r)w^{(s^{\prime },u)}(r^{\prime }) \\ 
& \quad \quad \quad 
+w^{(s,u)}(r)q(r^{\prime })+w^{(s,u)}(r)w^{(s,u)}(r^{\prime
})-w^{(s,u)}(r)w^{(s^{\prime },u)}(r^{\prime }) \\ 
& \quad \quad \quad \quad 
-w^{(s^{\prime },u)}(r)q(r^{\prime })-w^{(s^{\prime
	},u)}(r)w^{(s,u)}(r^{\prime })+w^{(s^{\prime },u)}(r)w^{(s^{\prime
},u)}(r^{\prime }).
\end{align*}
The left-hand-side and the  right hand side coincide due to the second Peiffer relation for the crossed module $\partial_{2}'.$ We switch from $s^{\prime }(r)$ to $s(r)+(\partial _{2}^{\prime }\circ q)(r)$ in all Peiffer liftings and then we use axiom 2XM5  of the definition of 2-crossed modules. 
\end{proof}

We are now ready to prove that:
\begin{theorem}\label{rwhisker2}
	{Consider a diagram of 2-crossed modules maps $f,g,h\colon \A\to \A'$, with $\A$ free up to order one,   homotopies and 2-fold homotopies, fitting into:
		$$\xymatrix{ f
			\ar@/^1.4pc/[rrr]^{(f,s,t)}="1"
			\ar@/_1.4pc/[rrr]_{(f,s^{\prime},t^{\prime})}="2"
			&&& g
			\ar@{}"1";"2"|(.2){\,}="7"
			\ar@{}"1";"2"|(.8){\,}="8"
			\ar@{=>}"7" ;"8"|{(f,s,t,q)}
			\ar[rr]^{(g,u,v)}
			&& h
		}.
		$$
		The quadratic $(f,s\boxplus u,t\boxplus v)$-2-derivation $(q\circledcirc u)\colon R \to L'$ yields a 2-fold homotopy connecting $(f,s\boxplus u,t\boxplus v)$ to $(f,s^{\prime}\boxplus u,t^{\prime}\boxplus v)$, thus fitting into:}
	\begin{equation*}
	\xymatrix{f
		\ar@/^1.4pc/[rrr]^{(f,s\boxplus u,t\boxplus v)}="1"
		\ar@/_1.4pc/[rrr]_{(f,s^{\prime}\boxplus u,t^{\prime}\boxplus v)}="2"
		&&& h
		\ar@{}"1";"2"|(.2){\,}="7"
		\ar@{}"1";"2"|(.8){\,}="8"
		\ar@{=>}"7" ;"8"|{\; (f,s\boxplus u,t\boxplus v,q\circledcirc u)}
	}.
	\end{equation*}
\end{theorem}
\begin{proof}
	The first equality in \eqref{source_target} is already proven; see  \eqref{1ndcond}. The second means:
	\begin{align*}
	(t^{\prime}\boxplus v)(e)=(t\boxplus v)(e)-\big((q\circledcirc u)\circ \partial _{2}^{\prime }\big)(e),
	\end{align*}
	for all $e\in E$. Let us prove it. Firstly we have:
	\begin{align*}
	(t^{\prime }\boxplus v)(e) & = t^{\prime }(e)+v(e)+(w^{(s^{\prime },u)}\circ \partial _{1})(e) \\
	& = t(e)-(q\circ \partial _{1})(e)+v(e)+(w^{(s^{\prime},u)}\circ \partial _{1})(e).
	\end{align*}
	On the other hand, by using the  form \eqref{formulawisk} for $(q\circledcirc u)$,   we obtain:
	\begin{align*}
	(t\boxplus v)(e)-\big((q\circledcirc u)\circ
	\partial _{1}\big)(e) & = t(e)+t^{\prime \prime }(e)+(w^{(s,u)}\circ \partial _{1})(e) \\
	& \quad \quad
	-(q(\partial _{1}(e))+w^{(s^{\prime },u)}(\partial _{1}(e)))-w^{(s,u)}(\partial _{1}(e)) \\
	& = t(e)+v(e)-\big(q(\partial _{1}(e))+w^{(s^{\prime},u)}(\partial _{1}(e))\big).
	\end{align*}
	This gives us the second equality in equation \eqref{source_target}.
\end{proof}
% 
% \begin{remark}
% It is clear that, the 2-quadratic derivation we formulated in Theorem \ref{rwhisker} is the unique derivation we mentioned in Idea \ref{uniquerwhisker} which 
% is also extended by $q$ and in the free basis element $b\in B$,
% \begin{align*}
% (q\circledcirc u)(b)=q(b)
% \end{align*}
% (See the properties of the map $w$ in ?? and the property of free algebra)
% \end{remark}

\begin{theorem}\label{actionright}
	{Let $\A$ and $\A'$ be 2-crossed modules, with $\A$ free up to order one, with a chosen basis.  Consider the groupoid  ${\rm HOM}_2(\A,\A')_1$ given by the disjoint union of all groupoids of 2-fold homotopies between fixed 2-crossed module maps. In other words: $${\rm HOM}_2(\A,\A')_1=\bigsqcup_{f,g \colon A\to \A'} {\rm HOM}_{2}(\A,\A',f,g).$$
		The right whiskering $  (q,u) \in {\rm HOM}_{2}(\A,\A')_{1} \times {\rm HOM}(\A,\A') \mapsto q\circledcirc u $
		yields  a (right) action of the groupoid ${\rm HOM}(\A,\A')$ of 2-crossed modules maps $\A\to \A'$ and their homotopies on the groupoid ${\rm HOM}_{2}(\A,\A')_{1}$, by groupoid maps.}
\end{theorem}
\begin{proof}
	What we need to prove is that, given quadratic 2-derivations $q$ and $q'$ and quadratic derivations $u$ and $m$, with the correct sources and targets, then: \begin{align*}
	(q+q^{\prime})\circledcirc u&=(q\circledcirc u)+(q^{\prime}\circledcirc u),\\
	q\circledcirc (u \boxplus m)&=(q\circledcirc u)\circledcirc m.
	\end{align*}
	Suppose we have homotopies and 2-fold homotopies, fitting in the diagram:
	\begin{align*}
	\xymatrix{ f \ar@/^2.4pc/[rrr]^{(f,s,t)}="1"
		\ar[rrr]|{(f,s^{\prime},t^{\prime})}="2"
		\ar@/_2.4pc/[rrr]_{(f,s^{\prime\prime},t^{\prime\prime})}="9"
		\ar@{}"1";"2"|(.2){\,}="7"
		\ar@{}"1";"2"|(.8){\,}="8"
		\ar@{=>}"7";"8"|{(f,s,t,q)}
		\ar@{}"2";"9"|(.2){\,}="10"
		\ar@{}"2";"9"|(.8){\,}="11"
		\ar@{=>}"10" ;"11"|{(f,s^{\prime},t^{\prime},q^{\prime})}
		&&& g
		\ar[rr]^{(g,u,v)}
		&& h }
	\end{align*}
	Then we have:
	\begin{align*}
	(q+q^{\prime })\circledcirc u & = q+q^{\prime}+w^{(s,u)}-w^{(s'',u)} \\
	& = (q+w^{(s,u)}-w^{(s',u)})+(q^{\prime }+w^{(s',u)}-w^{(s'',u)}) \\
	& = (q\circledcirc u)+(q^{\prime }\circledcirc u),
	\end{align*}
	thus right whiskering is functorial with respect to  vertical composition of 2-derivations.

	Suppose that we have  homotopies and 2-fold homotopies, as in the diagram:
	\begin{align*}
	\xymatrix{ f
		\ar@/^1.4pc/[rrr]^{(f,s,t)}="1"
		\ar@/_1.4pc/[rrr]_{(f,s^{\prime},t^{\prime})}="2"
		&&& g
		\ar@{}"1";"2"|(.2){\,}="7"
		\ar@{}"1";"2"|(.8){\,}="8"
		\ar@{=>}"7" ;"8"|{(f,s,t,q)}
		\ar[rr]^{(g,u,v)}
		&& h
		\ar[rr]^{(h,m,n)}
		&& k
	}.
	\end{align*}
	Let us prove that $q\circledcirc (u \boxplus m)=(q\circledcirc u)\circledcirc m.$
	The left-hand-side is:\begin{align*}
	q\circledcirc (u\boxplus m) = q+w^{(s,u\boxplus m)}-w^{(s^{\prime},u\boxplus m)},
	\end{align*}
	whereas the right-hand-side is:
	\begin{align*}
	(q\circledcirc u)\circledcirc m & = (q\circledcirc u)+w^{(s\boxplus u,m)}-w^{(s^{\prime }\boxplus u,m)} \\
	& = q+w^{(s,u)}-w^{(s^{\prime },u)}+w^{(s\boxplus u,m)}-w^{(s^{\prime}\boxplus u,m)}.
	\end{align*}
	By using equation \eqref{compatibility}, these two coincide.
\end{proof}

\subsection{Left whiskering of 2-fold homotopies by homotopies}
\noindent We now deal with left-whiskering. Let $f,g,h \colon \A \to \A'$ be 2-crossed module morphisms. Consider also homotopies $f \ra{(f,s,t)} g$,  $  g \ra{(g,u,v)} h$ and $g \ra{(g,u',v')} h$. 
Let also $(g,u,v,q^{\prime})$ be a 2-fold homotopy fitting into the diagram:
\begin{align*}
\xymatrix{f
	\ar[rr]^{(f,s,t)}
	&& g
	\ar@/^1.4pc/[rrr]^{(g,u,v)}="1"
	\ar@/_1.4pc/[rrr]_{(g,u^{\prime},v^{\prime})}="2"
	&&& h
	\ar@{}"1";"2"|(.2){\,}="7"
	\ar@{}"1";"2"|(.8){\,}="8"
	\ar@{=>}"7" ;"8"|{(g,u,v,q^{\prime})}
}
\end{align*}
For defining the pertinent left whiskering, we must  define a quadratic $(f,s\boxplus u,t\boxplus v)$-2-derivation connecting $(f,s\boxplus u,t\boxplus v)$ to $(f,s\boxplus u^{\prime},t\boxplus v^{\prime})$.
% with the diagram,%
% $$\xymatrix{f
% \ar@/^2pc/[rrrrr]^{(f,s\boxplus u,t\boxplus v)}="1"
% \ar@/_2pc/[rrrrr]_{(f,s\boxplus u^{\prime},t\boxplus v^{\prime})}="2"
% &&&&& h
% \ar@{}"1";"2"|(.2){\,}="7"
% \ar@{}"1";"2"|(.8){\,}="8"
% \ar@{=>}"7" ;"8"^{\; ?}
% }$$
As before, let $(s\circledcirc q^{\prime})$ be the unique quadratic $(f,s\boxplus u,t\boxplus v)$-2-derivation (see Remark \ref{uniquelyextended}) extending the function $q'_{|B}$ inside the diagram below (therefore $(s\circledcirc q^{\prime})(b)=q^{\prime}(b)$ for all $b\in B$):
\begin{align*}
\xymatrix@R=35pt@C=35pt{\free{B}{R}{Q'}{(f_{0},s+u,q_{|B}')}{(f_{0},s\boxplus u,s\circledcirc q^{\prime})}}
\end{align*}

The proof of the following result proceeds exactly as in the case of right-whiskering.
\begin{theorem}\label{lwhisker}
	{The quadratic $(f,s\boxplus u,t\boxplus v)$-2-derivation $(s\circledcirc q^{\prime})$ has the form $$(s\circledcirc q^{\prime})(r)=q^{\prime}(r)+w^{(s,u)}(r)-w^{(s,u^{\prime})}(r),$$
		and $(s\circledcirc q^{\prime})$ is a 2-fold homotopy connecting $(f,s\boxplus u,t\boxplus v)$ to $(f,s\boxplus u^{\prime},t\boxplus v^{\prime})$, therefore fitting into the diagram below:}
	\begin{align*}
	\xymatrix{f
		\ar@/^1.4pc/[rrr]^{(f,s\boxplus u,t\boxplus v)}="1"
		\ar@/_1.4pc/[rrr]_{(f,s\boxplus u^{\prime},t\boxplus v^{\prime})}="2"
		&&& h
		\ar@{}"1";"2"|(.2){\,}="7"
		\ar@{}"1";"2"|(.8){\,}="8"
		\ar@{=>}"7" ;"8"|{\; (f,s\boxplus u,t\boxplus v,s\circledcirc q^{\prime})}.
	}
	\end{align*}
	Moreover, left whiskering defines a left action of the groupoid ${\rm HOM}(\A,\A')$ of crossed module maps and their homotopies on the groupoid ${\rm HOM}_{2}(\A,\A')_{1}$, by groupoid maps. 
\end{theorem}
\begin{theorem}
	Right whiskering and left whiskering can be interchanged. Namely,  whenever all operations make sense we have:
	\begin{equation*}
	(s\circledcirc q)\circledcirc s^{\prime}=s\circledcirc (q\circledcirc s^{\prime}).
	\end{equation*}
	Therefore  diagrams of the following type can be unambiguously evaluated:
	\begin{align*}
	\xymatrix@R20pt@C=20pt{f
		\ar[rr]^{(g,s,t)}
		&& g
		\ar@/^1.4pc/[rrr]^{(g,u,v)}="1"
		\ar@/_1.4pc/[rrr]_{(g,u^{\prime},v^{\prime})}="2"
		&&& h
		\ar@{}"1";"2"|(.2){\,}="7"
		\ar@{}"1";"2"|(.8){\,}="8"
		\ar@{=>}"7" ;"8"|{(g,u,v,q)}
		\ar[rr]^{(h,s^{\prime},t^{\prime})}
		&& k
	}.
	\end{align*}
\end{theorem}
\begin{proof}
	If we use the property of $w$ given in equation \eqref{compatibility}, we get:
	\begin{align*}
	(s\circledcirc q)\circledcirc s^{\prime } & = \Big(
	q+w^{(s,u)}-w^{(s,u^{\prime })} \Big) \circledcirc s^{\prime } \\
	& = q+w^{(s,u)}-w^{(s,u^{\prime })}+w^{(s\boxplus u,s^{\prime
		})}-w^{(s\boxplus u^{\prime },s^{\prime })} \\
	& = q+w^{(u,s^{\prime })}-w^{(u^{\prime },s^{\prime })}+w^{(s,u\boxplus
		s^{\prime })}-w^{(s,u^{\prime }\boxplus s^{\prime })} \\
	& = s\circledcirc \Big( q+w^{(u,s^{\prime })}-w^{(u^{\prime
		},s^{\prime })}\Big)  = s\circledcirc (q\circledcirc s^{\prime }).
	\end{align*}
\end{proof}

{Thus, given two 2-crossed module $\A$ and $\A'$, with $\A$ free up to order one, and with a chosen basis, we have  a sesquigroupoid \cite{St}, denoted by ${\rm HOM}(\A,\A')_{2}$, whose objects are 2-crossed module maps $\A \to \A'$, the 1-morphisms are their homotopies, whereas 2-morphisms are the 2-fold homotopies between 2-crossed module homotopies.}

\subsection{The interchange law}
Consider 2-crossed modules $\A=(L,E,R,\partial _{1},\partial _{2})$, $\A'=(L^{\prime },E^{\prime
},R^{\prime },\partial _{1}^{\prime },\partial _{2}^{\prime })$, with $\A$ free up to order one with a chosen basis of $R$. For proving that the sesquigroupoid ${\rm HOM}(\A,\A')_{2}$ is a 2-groupoid, we must show that it satisfies the interchange law.
\begin{theorem}
	Suppose that we have 2-crossed module maps $f,g,h \colon \A \to \A'$, as well as homotopies and 2-fold homotopies, all fitting into the diagram:
	\begin{equation*}
	\xymatrix{ f
		\ar@/^1.4pc/[rrr]^{(f,s,t)}="1"
		\ar@/_1.4pc/[rrr]_{(f,s^{\prime},t^{\prime})}="2"
		&&& g
		\ar@{}"1";"2"|(.2){\,}="7"
		\ar@{}"1";"2"|(.8){\,}="8"
		\ar@{=>}"7" ;"8"|{(f,s,t,q)}
		\ar@/^1.4pc/[rrr]^{(g,u,v)}="3"
		\ar@/_1.4pc/[rrr]_{(g,u^{\prime},v^{\prime})}="4"
		&&& h
		\ar@{}"3";"4"|(.2){\,}="5"
		\ar@{}"3";"4"|(.8){\,}="6"
		\ar@{=>}"5" ;"6"|{(g,u,v,q^{\prime})},
	}\end{equation*}
	then the interchange law, below, holds:
	\begin{align*}
	(q\circledcirc u)\star (s^{\prime}\circledcirc q^{\prime})=(s\circledcirc q^{\prime})\star (q\circledcirc u').
	\end{align*}
	In fact, both sides of the previous equation are equal to: $$q+q^{\prime}+w^{(s,u)}-w^{(s^{\prime},u^{\prime})}.$$
	Therefore, we can define the horizontal composition of 2-fold homotopies as:
	\begin{equation*}
	\xymatrix{ f
		\ar@/^1.4pc/[rrr]^{(f,s,t)}="1"
		\ar@/_1.4pc/[rrr]_{(f,s^{\prime},t^{\prime})}="2"
		&&& g
		\ar@{}"1";"2"|(.2){\,}="7"
		\ar@{}"1";"2"|(.8){\,}="8"
		\ar@{=>}"7" ;"8"|{(f,s,t,q)}
		\ar@/^1.4pc/[rrr]^{(g,u,v)}="3"
		\ar@/_1.4pc/[rrr]_{(g,u^{\prime},v^{\prime})}="4"
		&&& h
		\ar@{}"3";"4"|(.2){\,}="5"
		\ar@{}"3";"4"|(.8){\,}="6"
		\ar@{=>}"5" ;"6"|{(g,u,v,q^{\prime})}
	}
	=\xymatrix@R=30pt@C=20pt{ f
		\ar@/^1.5pc/[rrrr]^{(f,s\boxplus u ,t\boxplus v)}="11"
		\ar@/_1.5pc/[rrrr]_{ (f,s' \boxplus u', t' \boxplus v') }="12"
		&&&& g
		\ar@{}"11";"12"|(.2){\,}="17"
		\ar@{}"11";"12"|(.8){\,}="18"
		\ar@{=>}"17" ;"18"|{(f,s\boxplus u ,t\boxplus v,q\otimes q')}
	},
	\end{equation*}
	where, by definition: 
	\begin{align*}
	(q\otimes q^{\prime})(r)=q(r)+q^{\prime}(r)+w^{(s,u)}(r)-w^{(s',u')}(r).
	\end{align*}
	Thus we have a 2-groupoid ${\rm HOM}(\A,\A')_{2}$  of 2-crossed module maps $\A \to \A'$,  their homotopies, and 2-fold homotopies between 2-crossed module homotopies.
\end{theorem}
\begin{proof}
	We have:
	\begin{align*}
	(q\circledcirc u)\star (s^{\prime }\circledcirc q^{\prime }) & = %
	\Big( q+w^{(s,u)}-w^{(s^{\prime },u)}\Big) \star \Big( q^{\prime
	}+w^{(s^{\prime },u)}-w^{(s^{\prime },u^{\prime })}\Big)  \\
	& = q+w^{(s,u)}-w^{(s^{\prime },u)}+q^{\prime }+w^{(s^{\prime
		},u)}-w^{(s^{\prime },u^{\prime })} \\
	& = q+w^{(s,u)}+q^{\prime }-w^{(s^{\prime },u^{\prime })} \\
	& = q^{\prime }+w^{(s,u)}-w^{(s,u^{\prime })}+q+w^{(s,u^{\prime
		})}-w^{(s^{\prime },u^{\prime })} \\
	& = \Big(q^{\prime }+w^{(s,u)}-w^{(s,u^{\prime })}\Big) \star \Big(
	q+w^{(s,u^{\prime })}-w^{(s^{\prime },u^{\prime })}\Big) \\
	& = (s\circledcirc q^{\prime })\star (q\circledcirc u^{\prime }).
	\end{align*}
	The rest follows immediately.
\end{proof}
\begin{remark}
	Note that all composition operations in the 2-groupoid ${\rm HOM}(\A,\A')_{2}$  involving 1-fold homotopies explicitly depend on the chosen free basis of $R$. 
\end{remark}
\begin{remark}
	In this paper and \cite{IJK}, we worked with free up to order one crossed modules $\A=(L,E,R,\partial _{1},\partial _{2})$ with a chosen basis of $R$, meaning that we chose an isomorphism $R\to \kappa[X]$, where $X$ is a set. All of the same construction with also work if $R$ was a free commutative algebra $S(V)$  (the symmetric algebra) over a vector space $V$, as long as we chose a particular isomorphism $R \to S(V)$.
\end{remark}

\end{document}